\documentclass[12pt]{article}

\usepackage{graphicx}%
\usepackage{multirow}%
\usepackage{amsmath,amssymb,amsfonts}%
\usepackage{amsthm}%
\usepackage{mathrsfs}%
\usepackage[title]{appendix}%
\usepackage{textcomp}%
\usepackage{manyfoot}%
\usepackage{booktabs}%
\usepackage{algorithm}%
\usepackage{algorithmicx}%
\usepackage{algpseudocode}%
\usepackage{listings}%
\usepackage{enumitem}
\usepackage{ulem}
\usepackage[dvipsnames]{xcolor}

\theoremstyle{thmstyleone}%
\newtheorem{theorem}{Theorem}
\newtheorem{lemma}[theorem]{Lemma}
\newtheorem{cor}[theorem]{Corollary}
\newtheorem{proposition}[theorem]{Proposition}%
\newtheorem{claim}{Claim}

\theoremstyle{thmstyletwo}%
\newtheorem{example}{Example}%

\theoremstyle{thmstylethree}%
\newtheorem{definition}{Definition}%

\begin{document}

\begin{center}
	{\Large Perturbation Method in Musielak-Orlicz Sequence Spaces}
\end{center}

\begin{center}
	{\sc Pando Georgiev},\\ 
	Institute of Mathematics and Informatics, Bulgarian Academy of Sciences, 
	8 Acad. Georgi Bonchev Str., Sofia, 1113, Bulgaria, pandogeorgiev2020@gmail.com\\
	{\sc Vasil Zhelinski},\\
	Faculty of Mathematics and Informatics, Sofia University St. Kliment Ohridski, 5 James Bourchier Blvd., Sofia, 1504, Bulgaria, v.zhelinski@gmail.com\\
	{\sc Boyan Zlatanov}, \\
	Faculty of Mathematics and Informatics, University of Plovdiv Paisii Hilendarski, 24 Tsar Assen Str., Plovdiv, 4000, Bulgaria, bobbyz@uni-plovdiv.bg
\end{center}

{\sc abstract:}We generalize an abstract variational principle in Banach spaces, introduced by Topalova \& Zlateva \cite{Topalova-Zlateva},
	by showing that the set $\mathbb{P}_0$ of perturbations for which a perturbed lower semi-continuous function $f$ is WPMC (Well Posed Modulus Compact) not only contains a dense $G_\delta$ subset, but is also a complement to a $\sigma$-porous subset in a specifically defined positive cone. 
	
	Moreover, if the space is a Musielak-Orlicz sequence space satisfying $\ell_\Phi\cong h_{\Phi}$, then the notion WPMC is replaced by the stronger notion of  Tikhonov well posedness, which is proved to be equivalent to the single-valuedness and upper semi-continuity of the multivalued mapping assigning a parameter to the solution set. 
	
	We give several applications. The first one is 
	that the Musielak-Orlicz sequence spaces have the Radon-Nikodym property and, therefore, are dentable by proving the validity of  Stegall's variational principle. As a consequence we obtain that the duals of Musielak-Orlicz sequence spaces are $w^*$-Asplund. We establish also a sufficient condition for Musielak-Orlicz and Nakano sequence spaces to be  Asplund spaces. 
	The next applications are for  determining the type of the smoothness of certain Musielak-Orlicz, Nakano, and weighted Orlicz sequence spaces. We illustrate by an example
	that it is possible to consider an Orlicz function without the $\Delta_2$ condition,   by a particular choice of the weighted sequence $\{w_n\}_{n=1}^\infty$ to get $\ell_M(w)\cong h_M(w)$ and to be able to apply the main result.

{\bf keywords:} Musielak-Orilcz sequence space, weighted Orlicz sequence spaces, Nakano sequence space, perturbation space, variational principle.


\section{Introduction}\label{sec1}
{
A new perturbation tool or variational principle in Orlicz sequence spaces is established by Topalova \& Zlateva in \cite{Topalova-Zlateva} (see Theorem 2). As an application, the authors found the best possible order of the reminder of a differentiable bump function in such spaces.

The variational principle in \cite{Topalova-Zlateva} is extended here (by adding a stronger condition) in two directions: 1) 
the  dense $G_\delta$ subset can be replaced by a complement of $\sigma$-porous set,  2)  if the space 
is a Musielak-Orlicz sequence space satisfying $\ell_\Phi\cong h_{\Phi}$, then the notion well-posed modulus compact (WPMC) minimization problem is replaced by the stronger Tikhonov well posedness, if the domain of the function under consideration is in the positive octane of the space. This is proved to be equivalent to the single-valuedness and upper semi-continuity of the multivalued mapping assigning a parameter to the solution set.

As an application, we prove the Stegall variatioinal principle in Musielak-Orlicz sequence spaces satisfying 
$\ell_\Phi\cong h_\Phi$, namely, showing that the set of linear perturbations determining Tichonov well posedness contains a dense $G_\delta$ subset. 
As a consequence we get that Musielak-Orlicz sequence spaces have Radon-Nikodym property and, therefore, are dentable and their dual spaces are $w^*$-Asplund spaces. With a different argument we establish a sufficient condition for Musielak-Orlicz and Nakano sequence spaces to be  Asplund spaces.  

Variational principles are widely used in getting of negative results in the renorming theory of Banach spaces \cite{Deville-Godefroy-Zizler}.
With the help of Ekeland variational principle upper estimates for the order differentiablity of norms are 
obtained in $\ell_p$ spaces \cite{Deville-Godefroy-Zizler}.  With the Stegall variational principle lower and upper estimates are obtained about for the order differentiablity of norms and bump functions in certain Banach spaces \cite{Hernandez-Troyanski}.
We give few examples how the method can be used for determining the type of the smoothness of certain Musielak-Orlicz, Nakano, and weighted Orlicz sequence spaces. We illustrate by an example
that it is possible to consider an Orlicz function without the $\Delta_2$ condition, but by a particular choice of the weighted sequence $\{w_n\}_{n=1}^\infty$ to get $\ell_M(w)\cong h_M(w)$ and to be able to apply the main result and its corollaries.

The article is organized as follows. 

In Section 2 we present some well known notions, notations, and results about
abstract perturbation technique, porosity and Musielak-Orlicz sequence spaces.

In section 3 we deliver some auxiliary results for $\sigma-$porous sets, define a perturbation space in a strong sense and prove a stronger abstract variational principle with respect to porosity of the exceptional set. 

We prove the following characterization of Musielak-Orlicz spaces: they satisfy $\ell_\Phi=h_\Phi$ if and only if their {\it modular} 
$$
\tilde{\Phi}(x):=\sum_{i=1}^\infty \Phi_i(x_i)
$$
 has a strong minimum at $0$.

A crucial lemma is proved here: the set of all {\it weighted modulars}
$$
\Big\{\widetilde \Phi_a(x):=\sum_{i=1}^\infty a_i \Phi_i(x_i): a\in \ell_\infty^+ \Big\}
$$
 forms a perturbation space in a strong sense.
 
 We establish a key homeomorphism between $\ell_\Phi^+$ and $\ell_1^+$, which facilitates the proofs.

In section 4 we prove the main result in two main theorems: a variational principle for Tikhonov well posedness in Musielak-Orlicz sequence spaces with respect to $\sigma$-porosity of the exceptional set.

In section 5 we apply the main result for proving the validity of Stegall's variational principle in Musielak-Orlicz sequence spaces. Therefore, such spaces possess the Radon-Nikodym property and their dual spaces are $w^*$-Asplund.

We present applications to differentiability of bump functions in Musielak-Orlicz sequence spaces, and
determine the type of the smoothness in Musielak-Orlicz, Nakano, and weighted Orlicz sequence spaces.
}

In section 6 we give examples of weighted Orlicz spaces without $\Delta_2$-condition for which the main results are applicable.

\section{Preliminaries}

We use the standard Banach space terminology from \cite{Lindenstrauss-Tzafriri}. Let
$(X,\|\cdot\|)$ be a normed space and $(X,\rho)$ be a metric spaces. Whenever we consider a metric $\rho$ in a normed space $(X,\|\cdot\|)$ we assume $\rho(x,y)=\|x-y\|$. We denote by $S_X$ be the unit sphere of $X$, and $B_X$ be the closed unit ball of $X$. 
We will denote by $\mathbb{N}$ and $\mathbb{R}$ the sets of the natural and real numbers, respectively.
For $K>0$ we will denote $KB_{(X,\|\cdot\|)}=\{x\in X:\|x\|\leq K\}$.

\subsection{Abstract Perturbation Technique and Porosity}

The main technique of \cite{Ivanov-Zlateva} was refined in \cite{Topalova-Zlateva}, in order to be applied in the context of Orlicz sequence spaces.

Let $(X,\|\cdot\|)$ be a Banach space and let $S$ be a nonempty subset in $X$. By $\alpha(S)$ the 
 index of non-compactness of $S$ is denoted (see \cite{Kuratowski}), that is the infimum of all $\varepsilon >0$ for which $S$ admits finite (or compact) $\varepsilon$ net.

For a function $f:X\to \mathbb{R}\cup \{+\infty\}$ and $\varepsilon\geq 0$ we will use the notation
$$
\Omega_f^S(\varepsilon)=\left\{x\in S:f(x)\leq\inf_{S}f+\varepsilon\right\}.
$$
Clearly, $\Omega_f^S(0) = {\rm argmin}_Sf$.

\begin{proposition}(\cite{Topalova-Zlateva})\label{proposition:1}
	For any $f,g:X\to \mathbb{R}\cup \{+\infty\}$, bounded below on $S$, and $\delta>0$ if
	$\Omega_f^S(\delta)\cap \Omega_g^S(\delta)\not=\emptyset$, then
	$\Omega_{f+g}^S(\delta)\subset \left(\Omega_f^S(3\delta)\cap \Omega_g^S(3\delta)\right)$. 	
\end{proposition}

\begin{definition}\label{definition:1}(\cite{Ivanov-Zlateva-13})
	Let $S\subset X$ and $f:S\to \mathbb{R}\cup \{+\infty\}$. We say that the minimization problem
	$(f,S)=\min\{f(x):x\in S\}$ is well-posed modulus compact, abbreviated as $WPMC$, if $\inf_S f\in \mathbb{R}$ and
	$\Omega_f^S(0)$ is nonempty and compact set such that if $\{x_n\}_{n=1}^\infty$ is a minimizing
	sequence: $\lim_{n\to\infty}f(x_n)=\min_S f$, then
	$\lim_{n\to\infty}d\left(x_n,\Omega_f^S(0)\right)=0$.
\end{definition}

The meaning of the term $WPMC$ is quite natural. If $\Omega_f^S(0)$ is a singleton,
then we get the well-known notion of Tikhonov well-posedness (see e.g. \cite{Dontchev-Zolezzi}, Chapter I). From a lemma of Kuratowski it follows \cite{Ivanov-Zlateva-13} that for a closed set $S\subset X$ and a proper lower semi-continuous function $f:S\to\mathbb{R}\cup\{+\infty\}$ bounded below on $S$ there holds
\begin{equation}\label{Kuratowski}
(f,S)\ \  \mbox{is}\ \  WPMC\ \ \mbox{if and only if}\ \ \lim_{t\downarrow 0}\alpha\left(\Omega_f^S(t)\right)=0.
\end{equation}

\begin{definition}\label{definition:2}(\cite{Topalova-Zlateva})
	Let $(X,\|\cdot\|)$ be a Banach space and $S\subseteq X$ be a non-empty set. The space $(\mathbb{P},\|\cdot\|_{\mathbb{P}})$ of real continuous, and bounded on $S$ functions, is called a perturbation space on $S$ if
	\begin{enumerate}[label=(d\ref{definition:2}.\roman*)]
		\item\label{lab:d2-1} $\mathbb{P}$ is a convex cone
		\item\label{lab:d2-2} $\mathbb{P}$ is complete with respect to the norm $\|\cdot\|_{\mathbb{P}}$,
		defined on ${\rm span}\{\mathbb{P}\}$, that dominates the uniform convergence on $S$, i.e.,   for some constant $C>0$ there holds 
		$$
		\sup_{x\in S}|f(x)|\leq C\|f\|_{\mathbb{P}}
		$$
		for all $f\in {\rm span}\{\mathbb{P}\}$
		(in other words, $({\rm span}\{\mathbb{P}\},\|\cdot\|_{\mathbb{P}})$ is a Banach space 
		of bounded on $A$ continuous, and $\mathbb{P}$ is its positive cone)
		\item\label{lab:d2-3} for every $\varepsilon>0$ there is $\delta=\delta(\varepsilon)>0$ so that for any $x\in S$ there exists $f\in \mathbb{P}$ (depending on $x$), so that 
		$$
		\|f\|_{\mathbb{P}}\leq \varepsilon,\ \ 
		x\in \Omega_{f}^S\left(\delta\right),\ \mbox{and}\ \ 
		\alpha(\Omega_{f}^S\left(3\delta\right))\leq \varepsilon.
		$$
	\end{enumerate}
	Then, we say that $(\mathbb{P},\|\cdot\|_{\mathbb{P}})$ is a perturbation space on $S$.
\end{definition}

\begin{theorem}\label{theorem:2}(\cite{Topalova-Zlateva}
	Let $(X,\|\cdot\|)$ be a Banach space, $S\subseteq X$ be nonempty and closed, $(\mathbb{P},\|\cdot\|_{\mathbb{P}})$ be a perturbation space on $S$, and $f:X\to \mathbb{R}\cup\{+\infty\}$ be a proper lower semi-continuous and bounded below on $S$ function.
	
	Then for every $\varepsilon>0$ there is $g\in \mathbb{P}$ so that $\|g\|_{\mathbb{P}}< \varepsilon$
	and $(f+g,S)$ is WPMC. In particular, $f+g$ attains its minimum on $S$.
\end{theorem}

We will modify the above theorem in the following way, involving porosity notions.

Let $(X,\rho)$ be a metric space. An open  (resp. closed) ball in $X$ with
center  $x$ and radius  $r>0$  is denoted by  $B(x,r)$  (resp. $\overline
B(x,r)$). If $(X,\|\cdot\|)$ is a normed space, then $B_X:=\overline{B}(0,1)$ (resp. $S_X:=\{x\in X: \|x\|=1\}$).

\begin{definition}\label{definition:3}(\cite{Phelps})
	A subset $A$ of a metric space $(X,\rho)$ is said to be porous in $X$, if there exist  $\alpha >0$ and $r_0 >0$ such that 
	for every $r\in (0,r_0)$ and $x \in  A$ there is $z\in  X$ such that $B(z,\alpha r)\subset B(x,r)\setminus A$.  
\end{definition}

\begin{definition}\label{definition:4}(\cite{Phelps})
	A set  $A$ is called  $\sigma$-porous in $(X\rho)$, if it is a countable union of sets, which are porous in $X$.
\end{definition}

Clearly, a set which is $\sigma $-porous in $ X$ is of first category, the converse being false in general. 

Given a normed space $E$, for $v\in E$ and $c>0$ define the cone
$$
C(v,c)=\{x\in E: x=\lambda v+w, \lambda>0, \| w\| <c\lambda \} = \bigcup _{\lambda>0}\lambda {B(v,c).}
$$

\begin{definition}\label{definition:5}(\cite{Ge})
	A subset $M$ of a normed space $(X,\|\cdot\|)$ is called cone porous (resp. fully cone porous) at $x\in M$, if there 
	exist $r>0, v\in X, \|v\| =1$ and $c\in (0,1)$ such that 
	$$
	M\cap B(x,r)\cap \big(x+ {C(v,c)}\big)=\emptyset,
	$$
	(resp. $M\cap  \big(x+ {C(v,c)}\big) = \emptyset$).
\end{definition}

\begin{definition}\label{definition:6}(\cite{Ge})
	The set $M$ is called cone porous (resp. fully cone porous), if it is cone porous (resp. fully cone porous) at every of its point. 
\end{definition}

\begin{definition}\label{definition:7}(\cite{Ge})
	$M$ is called $\sigma-$ cone porous (resp. $\sigma$-fully cone porous), if it is a union of countably many cone porous (resp. $\sigma$-fully cone porous) sets.
\end{definition}

It is clear that every fully cone porous set is cone porous and every cone porous set is porous. 

\subsection{Musielak-Orlicz Sequence Spaces}

Let
$\ell^0$ stand for the space of all real sequences, i.e.,
$x=\{x_i\}_{i=1}^\infty\in\ell^0$.

\begin{definition}\label{definition:8}(\cite{Cui-Hudzik})
	A Banach space $(X,\|\cdot\|)$ is said to be K\"othe sequence
	space if $X$ is a subspace of $\ell^0$ such that
	\begin{enumerate}[label=(d\ref{definition:8}.\roman*)]
		\item\label{lab:d5-1}if $x\in\ell^0$, $y\in X$ and $|x_i|\leq |y_i|$ for all $i\in\mathbb{N}$
		then $x\in X$ and $\|x\|\le\|y\|$
		\item\label{lab:d5-1}there exists an element $x\in X$ such that $x_i>0$ for all
		$i\in\mathbb{N}$.
	\end{enumerate}
\end{definition}

Following \cite{Lindenstrauss-Tzafriri}, a sequence $\{v_i\}_{i=1}^\infty$ in a Banach space $X$ is called
Schauder basis of $X$ (or basis for short) if for each $x\in X$
there exists an unique sequence $\{a_i\}_{i=1}^\infty$ of scalars
such that $x=\displaystyle\sum_{i=1}^\infty a_iv_i$ and we will use the notation ${\rm supp}\{x\}=\{i\in\mathbb{N}:a_i\not= 0\}$. If
$\{v_i\}_{i=1}^\infty$ is a basis in $X$ such that the series
$\displaystyle\sum_{i=1}^\infty a_iv_i$ converges whenever
$\displaystyle\sup_{n\in\mathbb{N}}\left\|\displaystyle\sum_{i=1}^n a_iv_i\right\|<\infty$,
then it is called a boundedly complete basis of $X$. A sequence of
non zero vectors $\{x^{(n)}\}_{n=1}^\infty$ of the form
$\displaystyle\sum_{i=p_n+1}^{p_{n+1}}a_iv_i$, with $\{a_i\}_{i=1}^\infty$
scalars and $0=p_1<p_2<p_3\dots$ an increasing sequence of
integers is called a block basic sequence or block basis of
$\{v_i\}_{i=1}^\infty$ for short. By $\{e_i\}_{i=1}^\infty$ we
denote the unit vectors of $\ell^0$.

Let us recall that an Orlicz function $M:\mathbb{R}\to\mathbb{R}$ is an even, continuous, convex function
such that $M$ is nondecreasing on $[0,+\infty)$ and $M(0)=0$. We say that $M$
is non--degenerate Orlicz function if $M(t)>0$ for every $t>0$. 

A sequence $\Phi =\{\Phi_{i}\}_{i=1}^\infty$ of Orlicz functions is
called a Musielak--Orlicz function (of two variables).

 We say that a Musielak-Orlicz function $\Phi$
is non--degenerate one if each $\Phi_n$ are non--degenerate Orlicz functions.

The Musielak--Orlicz sequence space $\ell_\Phi$, generated by a Musielak-Orlicz function $\Phi$
is the set of all real sequences $\{x_i\}_{i=1}^\infty$ such that
$\sum_{i=1}^\infty\Phi_i(\lambda x_i)<\infty$ for some $\lambda >0$.
The space $\ell_\Phi$ is a Banach space if endowed with
the Luxemburg's norm:
$$
\| x\|_{\Phi} =\inf\left\{
r>0:\sum_{i=1}^\infty\Phi_{i}(x_i/r)\leq 1\right\}.
$$

Throughout this note we always denote by $M$ an Orlicz function
and by $\Phi$ a Musielak-Orlicz function.

If the Musielak--Orlicz function $\Phi$ consists of one and the same Orlicz
function $M$ we get the Orlicz sequence space denoted by $\ell_M$.

A weight sequence $w=\{w_i\}_{i=1}^\infty$ is a sequence of
positive reals. We will distinguish two classes of weighted
sequences $\Lambda_\infty$ and $\Lambda$. The weight sequence
$w=\{w_i\}_{i=1}^\infty$ is from the class $\Lambda_\infty$ if it
is nondecreasing sequence with $\displaystyle\lim_{i\to\infty}w_i=\infty$.
Following \cite{Hernandez,Fuentes-Hernandez} we say that $w=\{w_i\}_{i=1}^\infty$ is from
the class $\Lambda$ if there exists a subsequence
$w=\{w_{i_k}\}_{k=1}^\infty$ such that
$\displaystyle\lim_{k\to\infty}w_{i_k}=0$ and $\displaystyle\sum_{k=1}^\infty
w_{i_k}=\infty$.

A weighted Orlicz sequence space $\ell_M(w)$ generated by an
Orlicz function $M$ and a weight sequence $w$ is the Musielak--Orlicz sequence
space $\ell_\Phi$, where $\Phi_i(t)=w_iM(t)$.

Weighted Orlicz sequence spaces were investigated for example in
\cite{Fuentes-Hernandez,Pierats-Ruiz,Kaminska98,Kaminska-Mastylo}. Let us
mention that if the weight sequence is from the class $\Lambda$,
then a lot of the properties of the space $\ell_M(w)$ depend only
on the generating Orlicz function $M$, which is in contrast with
the results when $w\not\in\Lambda$ \cite{Hernandez,Ruiz,Maleev-Zlatanov}.

It is well known that the Orlicz and the weighted Orlicz sequence
spaces equipped with Luxemburg norm are K\"othe
sequence spaces. 

For simplicity of notations we will use 
$$
\widetilde \Phi
(x)=\displaystyle\sum_{i=1}^\infty \Phi_i(x_i)
\ \ {\mbox and} \ \  \widetilde
M_w(x)=\displaystyle\sum_{i=1}^\infty w_iM(x_i)
$$
called {\it modular} of the space.
 
For every $a\in\ell_\infty$, $x\in \ell^{0}$ we will use the notation
$$
\widetilde{\Phi}_a(x)=\sum_{i=1}^{\infty}a_i\Phi_i(x_i)
$$
and will call it {\it weighted modular} (depending on $a$).

If $\{a_i\}_{i=1}^\infty$ is the constant sequence $a_i=1$ for all $i\in\mathbb{N}$, we get 
$\widetilde{\Phi}(x)=\widetilde{\Phi}_{\{1\}_{i=1}^\infty}(x)$.

An extensive study of Orlicz and Musielak--Orlicz spaces can be found in
\cite{Lindenstrauss-Tzafriri,Musielak}.

We denote by $h_\Phi$ (called {\it heart})  the closed linear subspace of $\ell_\Phi$,
generated by all $x \in \ell_\Phi$, such that $\widetilde\Phi (\lambda
x_i)<\infty$ for every $\lambda >0$ and by $h_M(w)$ the subspace of
$\ell_M(w)$ such that $\widetilde M_w(\lambda x)<\infty$ for every
$\lambda >0$.

The unit vectors $\{e_i\}_{i=1}^\infty$ form a boundedly complete
basis in $h_\Phi$, equipped with the Luxemburg norm.

We say that an Orlicz function $M$ has $\Delta_2$--condition if there exist $C>1$ and
$t_0>0$ such that $M(2t)\le CM(t)$ for every $t\in (0,t_0]$.

If $w\in\Lambda$ then the spaces $\ell_M(w)$ and $h_M(w)$ coincide
if and only if $M\in\Delta_2$. The proof is similar to that done in (\cite{Lindenstrauss-Tzafriri}
Proposition 4.a.4).

For a given Orlicz function $M$ the following numbers are associated
(see \cite{Lindenstrauss-Tzafriri}, p. 143):
$$
\alpha_M=\sup\{p:\sup\{ M(uv)/u^pM(v):u,v\in(0,1]\}>0\},
$$
$$
\beta_M=\inf\{p:\inf\{ M(uv)/u^pM(v):u,v\in(0,1]\}>0\}.
$$
An Orlicz function $M$ satisfies the $\Delta_2$--condition if and only if
$\beta_M<\infty$, which implies of course $M(uv)\geq u^qM(v)$,
$u,v\in [0,1]$ for some $q\geq \beta_M$ (see \cite{Lindenstrauss-Tzafriri} p.140).

\begin{definition}\label{definition:9}(\cite{Kaminska})
	We say that the Musielak--Orlicz function $\Phi$ satisfies the $\delta_2$--condition
	if there exist constants $K,\beta >0$ and a non--negative sequence
	$\{c_n\}_{n=1}^\infty\in\ell_1$ such that for every $n\in\mathbb{N}$
	\begin{equation}\label{equation:1} 
		\Phi_n(2t)\leq K\Phi_n(t)+c_n, 
	\end{equation} 
	provided $t\in [0,\Phi_n^{-1}(\beta)]$.
\end{definition}
The spaces $\ell_\Phi$ and $h_\Phi$ coincide if and only if $\Phi$ has
$\delta_2$--condition {\cite{Kaminska82}.

We say that the Musielak--Orlicz function $\Phi$ satisfies the uniform
$\delta_2$--condition if it satisfies (\ref{equation:1}) for every $t\in
[0,t_0]$ for some $t_0>0$ with $c_n=0$ for every $n\in\mathbb{N}$.

Recall that for given Musielak--Orlicz functions $\Phi$ and $\Psi$ the spaces
$\ell_\Phi$ and $\ell_\Psi$ coincide with equivalence of norms if and only if
$\Phi$ is equivalent to $\Psi$ at $0$, that is there exist constants
$K,\beta >0$ and a non--negative sequence $\{c_n\}_{n=1}^\infty\in\ell_1$,
such that for every $n\in\mathbb{N}$ the inequalities
$$ 
\Phi_n(Kt)\le \Psi_n(t)+c_n \ \ \mbox{and} \ \
\Psi_n(Kt)\le \Phi_n(t)+c_n 
$$
hold for every $t\in [0,\min(\Phi_n^{-1}(\beta),\Psi_n^{-1}(\beta))]$,
\cite{Kaminska,Maleev-Zlatanov-02}.

If $\Phi_i(1)=1$ for every $i\in\mathbb{N}$, then the unit vectors
$\{e_i\}_{i=1}^\infty$ is a normalized boundedly complete basis in
$h_\Phi$. If $\Phi_i(1)\not=1$, and $a_i$ is the solution of the
equation $\Phi_i(a_i)=1$, $i\in\mathbb{N}$, then $\ell_\Phi$, equipped with
the Luxemburg norm, is isometric to $\ell_\phi$, where
the Musielak--Orlicz function $\phi$ is defined by the sequence
?$\phi_i(t)=\displaystyle\frac{\Phi_i(a_it)}{\Phi_i(a_i)}$?. Therefore the sequence
$v_i=a_ie_i$, $i\in\mathbb{N}$ is a normalized boundedly complete basis in
$h_\Phi$. The unit vector basis $\{e_i\}_{i=1}^\infty$ is a
boundedly complete basis in $\ell_\phi$ if and only if $\ell_\phi\cong h_\phi$.
The weighted Orlicz sequence space $\ell_M(w)$ is isometric to
$\ell_\phi$, where the Musielak-Orlicz function $\phi$ is defined by
$$
\phi_i(t)=\displaystyle\frac{M(a_it)}{M(a_i)},\ \ \
a_i=M^{-1}(1/w_i).
$$ 
Hence $\ell_M(w)\cong h_M(w)$ if and only if
$\ell_\phi\cong h_\phi$ and $y=\sum_{i=1}^\infty x_ie_i\in h_\phi$
if and only if $x=\sum_{i=1}^\infty a_ix_ie_i\in h_M(w)$. The spaces
$\ell_M(w)$ and $\ell_\phi$ are isometric.

For any Orlicz function $M$ the function $M^{*}(t)=\sup_{s\geq 0}\{st-M(s)\}$ is called a conjugate 
function to $M$. It plays a crucial role in the investigation of the dual space $\ell_M^{*}$. 
For a Musielak-Orlicz function $\Phi=\{\Phi_n\}_{n=1}^\infty$ the conjugated function is
$\Phi^{*}=\{\Phi_n^{*}\}_{n=1}^\infty$. It is well known that if $\Phi\in\delta_2$, then $\ell_{\Phi}^{*}$ is isomorphic to $\ell_{\Phi^{*}}$. Moreover $\ell_{\Phi}$ is reflexive if $\Phi,\Phi^{*}\in\delta_2$ \cite{Kaminska-Lee}.

\subsection{Fr\'echet Differentiability in Musielak-Orlicz Sequence Spaces}

The existence of Fr\'echet smooth norms and bump
functions and its impact on the geometrical properties of a Banach
space have been subject to many investigations \cite{Deville-Godefroy-Zizler}. 
It is well known that the usual norm in $\ell_p$ and $L_p$ is $p-1$ times Fr\'echet differentiable for $p$ an odd number, $[p]$ times Fr\'echet differentiable for $p\not\in\mathbb{N}$, and infinitely many times Fr\'echet differentiable for $p$ an even number \cite{Bonoc-Frampton,Sundaresan}

An extensive study and bibliography may be found in \cite{Deville-Godefroy-Zizler}.

The exact order of differentiability in Orlicz spaces have been found in \cite{Maleev-Troyanski-91,Maleev-Troyanski-89,Maleev}.

We will investigate only Fr\'echet differentiable norms and bump functions, and for the sake 
of space we will refer sometimes it just as differentiable or smooth ones. 

As a Fr\'echet smooth norm immediately produces a bump function of
the same smoothness, all negative results about bump functions are
negative results about the smoothness in the class of all
equivalent norms.

Let $X$ be a Banach space and let $f:X\to\mathbb{R}$. The function $f$ is called Fr\'echet differentiable, or
a Fr\'echet smooth one at $x\in X$, if the following
representation holds:
$$
f(x+y)=f(x)+f^{\prime}(x)(y)+R(x,y)
$$
for every $y\in X$, where $f^{\prime}(x)$ is a linear bounded
form on $X$ and
$$
\lim_{\|y\|\to 0}\frac{|R(x,y)|}{\|y\|}=0.
$$
A norm $\|\cdot\|$ is refered as a Fr\'echet differentiable if it is 
smooth at all points $x\in S_X$, the unit sphere of $X$.

Sometimes it is interesting and applicable to know the best possible 
order of the reminder $R(x,y)$ and thus we get the next notion.
We shall say that $f$ is 
$F_{\omega}$--smooth at $x\in X$ for some $\omega :(0,1]\to\mathbb{R}^{+}$, with $\lim_{t\to
	0}\frac{\omega (t)}{t^p}=0$, $p\in [1,2]$, if
$$
\lim_{\|y\|\to 0}\frac{|R(x,y)|}{\omega(\|y\|)}=0.
$$

Let $\Phi =\{\Phi_n\}_{n=1}^\infty$ be a Musielak--Orlicz function. Let us define
by 
$$
\psi_n(t)=\displaystyle\int_0^t\frac{\Phi_n (u)}{u}du
$$ 
and
$$
\Psi_n(t)=\displaystyle\int_0^t\displaystyle\frac{\psi_n(u)}{u}\exp \frac{u}{u-t}du.
$$

Following \cite{Maleev-Zlatanov-02}, to every Musielak--Orlicz function $\Phi =\{\phi_n\}_{n=1}^\infty$
we associate the numbers 
$$
\alpha_\Phi =\liminf_{n\to\infty}\alpha_{\Phi_n}\ \mbox{and}\
\beta_\Phi =\limsup_{n\to\infty}\beta_{\Phi_n}.
$$

\begin{lemma}\label{lemma:3}(\cite{Maleev-Zlatanov-02})
	Let $\Phi =\{\Phi_n\}_{n=1}^\infty$ be a Musielak--Orlicz function.
	Then 
	\begin{enumerate}
		\item  $\Phi_n$ is equivalent to $\Psi_n$ at $0$
		\item $t^k|\Psi_n^{(k)}(t)|\le c_k\Psi_n(c_kt)$, $t\in
		[0,\infty )$, $c_k>0$, $k=1,2,\dots$
		\item $\ell_\Psi$ is isomorphic to $\ell_\Phi$
		\item $\alpha_\Psi=\alpha_\Phi$.
	\end{enumerate}
	for every $n$.
\end{lemma}

For $p\geq 1$ define $E(p)=\left\{
\begin{array}{ll}
	p-1,&p\in\mathbb{N} ,\\[12pt]
	[p],&p\not\in\mathbb{N} .
\end{array}
\right.$ 

\begin{theorem}\label{theorem:4}(\cite{Maleev-Zlatanov-02})
	Let $\Psi$ be a Musielak-Orlicz function and $1\le k=E(\alpha_\Psi )$.
	Then there exists an equivalent $k$--times Fr\'echet
	differentiable norm in $h_\Psi$. If in addition $\beta_\Phi<+\infty$ then there exists
	an equivalent $k$-times uniformly Fr\'{e}chet differentiable norm in $\ell_\Phi$.
\end{theorem}

Following (\cite{Deville-Godefroy-Zizler}, p. 58), a direct consequence of Theorem \ref{theorem:4} is the next corollary. 

\begin{cor}\label{corollary:5}
	Let $\Psi$ be a Musielak-Orlicz function and $1<\alpha_\Psi$,
	then $\ell_\Psi$ is an Asplund space.
\end{cor}

\section{Auxiliary Results}

To ease the proofs of the main results, we will present some intermediate ones in this section. Some of these results have their own value and can be applied in subsequent research, and this is another reason to formulate them separately instead of appearing as steps in the proofs of the main theorems.

\subsection{Auxiliary Results About $\sigma$-Porous Set}

We will need the next definition which is a generalization of Definition \ref{definition:2} introduced in \cite{Topalova-Zlateva}.

\begin{definition}\label{definition:11}
	We will say that the space $(\mathbb{P},\|\cdot\|_{\mathbb{P}})$ of real continuous, and bounded on $S$ functions, is a {\it perturbation space in a strong sense} on $S$ if
	the conditions \ref{lab:d2-1}, \ref{lab:d2-2}, and 
		
	\begin{enumerate}[label=(dS.ii\roman*)]
		\item\label{lab:d8-3} for any $n\in\mathbb{N}$ there exists $K_n\in (0,6C)$ so that for any $\varepsilon>0$,  for any $x\in S$ there exists $g\in \mathbb{P}$ (de\-pend\-ing on $x$ and $\varepsilon$), such that for $\delta=K_n\varepsilon$, 
		\begin{equation}\label{equation:4}
			\|g\|_{\mathbb{P}}\leq \varepsilon,\ \ 
			x\in \Omega_{g}^S\left(\delta\right),\ \mbox{and}\ \ 
			\alpha(\Omega_{g}^S\left(3\delta\right))\leq {1\over n}
		\end{equation}
	\end{enumerate}
	are satisfied.
\end{definition}

\begin{theorem}\label{theorem:7} 
	Let $(X,\|\cdot\|)$ be a Banach space, $S\subseteq X$ be nonempty and closed, $(\mathbb{P},\|\cdot\|_{\mathbb{P}})$ be a perturbation space on $S$ in the strong sense.	
	Let $f : X \rightarrow \mathbb{R} \cup \{+\infty\}$ be a proper, lower semi-continuous and bounded below on $S$ function. Then the set 
	$$
	A=\Big\{h\in \mathbb{P}: \mbox{the minimization problem}\ \  (f+h,S) \mbox { is not WPMC}\Big\}
	$$
	is a $\sigma$-porous set.
\end{theorem}

{
\begin{proof}
For each $n\in\mathbb{N}$ define

\[
A_n := \Big\{ h\in\mathbb{P} : \exists\, t>0 \text{ such that }
\alpha\big(\Omega_{f+h}^S(t)\big) < \tfrac1n \Big\}.
\]
By the lemma of Kuratowski mentioned above, see (\ref{Kuratowski}), it follows

\[
(f+h,S)\ \text{is WPMC}
\quad\Longleftrightarrow\quad
\forall n\ \exists t>0:\ \alpha(\Omega_{f+h}^S(t))<\tfrac1n,
\]
hence

\[
A:=\Big\{h:(f+h,S)\text{ is WPMC}\Big\}=\bigcap_{n=1}^\infty A_n.
\]
Therefore
\[
P_0:=\{h:(f+h,S)\text{ is not WPMC}\}
=\bigcup_{n=1}^\infty (\mathbb{P}\setminus A_n).
\]
Thus it suffices to show that $\mathbb{P}\setminus A_n$ is porous for each $n$.

\medskip
\noindent\textbf{Step 1. Choosing $x$ and $g$.}
Fix $n\in\mathbb{N}$, $h\in\mathbb{P}$, and $\varepsilon>0$.  
Let $K_n\in(0,6C)$ be the constant from condition \ref{lab:d8-3}, and set

\[
\delta := K_n\varepsilon.
\]
Choose any point
\[
x\in \Omega_{f+h}^S(\delta).
\]
By condition \ref{lab:d8-3}, applied to this same $x$, there exists $g\in\mathbb{P}$ with
\[
\|g\|_{\mathbb{P}}\le \varepsilon,
\qquad
x\in\Omega_g^S(\delta),
\qquad
\alpha(\Omega_g^S(3\delta))\le \tfrac1n.
\]
Thus

\[
x\in \Omega_{f+h}^S(\delta)\cap \Omega_g^S(\delta),
\]
so Proposition~\ref{proposition:1} applies to the pair $(f+h,g)$.

\medskip
\noindent\textbf{Step 2. Stability claim.}
\begin{claim}
Let $\phi:S\to\mathbb{R}$ be proper and bounded below, and let $u:S\to\mathbb{R}$ satisfy $\|u\|_\infty<\delta/3$. Then

\[
\Omega_{\phi+u}^S(\delta/3)\subset \Omega_\phi^S(\delta).
\]

\end{claim}

\begin{proof}[Proof of Claim]
If $x\in\Omega_{\phi+u}^S(\delta/3)$, then

\[
\phi(x)
< \phi(x)+u(x)+\tfrac{\delta}{3}
\le \inf_S(\phi+u)+\tfrac{2\delta}{3}
< \inf_S\phi+\delta.
\]

Thus $x\in\Omega_\phi^S(\delta)$.
\end{proof}

\medskip
\noindent\textbf{Step 3. Applying Proposition~\ref{proposition:1}.}
Take any $u\in\mathbb{P}$ with $\|u\|_{\mathbb{P}}<\delta/(3C)$, so $\|u\|_\infty<\delta/3$ by \ref{lab:d2-2}.  
Applying the Claim with $\phi=f+h+g$ gives

\[
\Omega_{f+h+g+u}^S(\delta/3)
\subset \Omega_{f+h+g}^S(\delta).
\]

Since $x\in\Omega_{f+h}^S(\delta)\cap\Omega_g^S(\delta)$,  
Proposition~\ref{proposition:1} yields

\[
\Omega_{f+h+g}^S(\delta)
\subset \Omega_{f+h}^S(3\delta)\cap \Omega_g^S(3\delta)
\subset \Omega_g^S(3\delta).
\]

Hence

\[
\Omega_{f+h+g+u}^S(\delta/3)\subset \Omega_g^S(3\delta),
\]

and therefore

\[
\alpha\big(\Omega_{f+h+g+u}^S(\delta/3)\big)
\le \alpha(\Omega_g^S(3\delta))
\le \tfrac1n.
\]

Thus
\begin{equation}\label{eq:hole}
h+g+\big(\mathbb{P}\cap B(0,\delta/(3C))\big)\subset A_n.
\end{equation}

\medskip
\noindent\textbf{Step 4. Constructing a porosity ball.}
Choose $x_0\in\operatorname{int}\mathbb{P}$ with $\|x_0\|=1$, and let $r_0\in(0,1)$ satisfy
$B(x_0,r_0)\subset\operatorname{int}\mathbb{P}$.  
Then

\[
\frac{\delta}{6C}x_0 + \frac{r_0\delta}{6C}B(0,1)
\subset \frac{\delta}{3C}B(0,1).
\]

Since $\|g\|_{\mathbb{P}}\le\varepsilon$ and $\delta/(3C)<2\varepsilon$, we obtain

\[
B\Big(h+g+\tfrac{\delta}{6C}x_0,\ \tfrac{r_0\delta}{6C}\Big)
\subset B(h,3\varepsilon).
\]

By \eqref{eq:hole},

\[
B\Big(h+g+\tfrac{\delta}{6C}x_0,\ \tfrac{r_0\delta}{6C}\Big)
\subset A_n.
\]

\medskip
\noindent\textbf{Step 5. Conclusion.}
The radius of this ball is

\[
\frac{r_0\delta}{6C}
=\frac{r_0K_n}{6C}\varepsilon,
\]
so the porosity constant is

\[
\frac{\frac{r_0K_n}{6C}\varepsilon}{3\varepsilon}
=\frac{r_0K_n}{18C},
\]
independent of $h$.  
Thus $\mathbb{P}\setminus A_n$ is porous.  
Since

\[
P_0=\bigcup_{n=1}^\infty (\mathbb{P}\setminus A_n),
\]
the set $P_0$ is $\sigma$-porous.
\end{proof}
}

 
\subsection{Auxiliary Results About Musielak-Orlicz Sequence Spaces}

\begin{lemma}\label{lemma:8A}
Let $\Phi$ be a non--degenerate Musielak--Orlicz function.  
Suppose there exists a sequence $\{x^{(n)}\}_{n=1}^\infty \subset \ell_\Phi$ such that

\[
I := \inf_{n\in\mathbb N} \|x^{(n)}\|_\Phi > 0,
\]

and

\[
\lim_{n\to\infty} e_\Phi(x^{(n)})
= \lim_{n\to\infty} \sum_{i=1}^\infty \Phi_i\!\left(x^{(n)}_i\right)
= 0.
\tag{9}
\]

Then $h_\Phi \not\cong \ell_\Phi$.  
In particular, there exists $x\in \ell_\Phi$ such that $x\notin h_\Phi$.
\end{lemma}

\begin{proof}
{
Since $\|x^{(n)}\|_\Phi \ge I$, by the definition of the Luxemburg norm,

\[
e_\Phi\!\left(\frac{x^{(n)}}{I}\right)
= \sum_{i=1}^\infty \Phi_i\!\left(\frac{x^{(n)}_i}{I}\right)
\ge 1,
\qquad n\in\mathbb N.
\tag{10}
\]

We consider two cases.

\medskip
\noindent
\textbf{Case I.}  
Assume that for every $i\in\mathbb N$,

\[
\lim_{n\to\infty}
\Phi_i\!\left(\frac{x^{(n)}_i}{I}\right) = 0.
\]

\medskip
\textbf{Step 1: Small initial blocks.}
Fix $m\in\mathbb N$.  
Since each term tends to $0$, the finite sum

\[
\sum_{i=1}^m \Phi_i\!\left(\frac{x^{(n)}_i}{I}\right)
\to 0.
\]

Thus there exists $N(m)$ such that for all $n\ge N(m)$,

\[
e_\Phi\!\left(\sum_{i=1}^m \frac{x^{(n)}_i}{I} e_i\right)
\le \frac14.
\tag{11}
\]

\medskip
\textbf{Step 2: Small tails.}
For each $n$, since $e_\Phi(x^{(n)}/I)\ge 1$ and finite,  
there exists $p_n\in\mathbb N$ such that

\[
\sum_{i=p_n+1}^\infty 
\Phi_i\!\left(\frac{x^{(n)}_i}{I}\right)
\le \frac14.
\tag{13}
\]

\medskip
\textbf{Step 3: A large middle block.}
Let $N_1,N_2\in\mathbb N$.  
Set $m=N_2-1$ and choose $n\ge \max\{N_1,\,N(m)\}$.

Then:

- by (11), the block $1,\dots,N_2-1$ contributes at most $1/4$,
- by (13), the tail $p_n+1,\dots,\infty$ contributes at most $1/4$,
- by (10), the whole vector contributes at least $1$.

Hence the middle block satisfies

\[
e_\Phi\!\left(\sum_{i=N_2}^{p_n} \frac{x^{(n)}_i}{I} e_i\right)
\ge \frac12.
\tag{14}
\]

\medskip
\textbf{Step 4: Choosing blocks with small unscaled modular.}
From (9), $e_\Phi(x^{(n)})\to 0$.  
Thus for each $n,N$ there exists $j_{N,n}$ such that

\[
e_\Phi\!\left(\sum_{i=N}^{p_{j_{N,n}}} x^{(j_{N,n})}_i e_i\right)
\le \frac{1}{2n},
\]

while by (14),

\[
e_\Phi\!\left(\sum_{i=N}^{p_{j_{N,n}}}
\frac{x^{(j_{N,n})}_i}{I} e_i\right)
\ge \frac12.
\tag{15}
\]

\medskip
\textbf{Step 5: Diagonal block construction.}
Define $q_1=0$ and recursively

\[
q_{n+1} := p_{\,j_{q_n+1,n}}.
\]

Then $q_1<q_2<q_3<\cdots$.

Define

\[
x^\ast := 
\sum_{n=1}^\infty 
\sum_{i=q_n+1}^{q_{n+1}}
x^{(j_{q_n+1,n})}_i\, e_i.
\]

Since the blocks are disjoint,

\[
e_\Phi(x^\ast)
= \sum_{n=1}^\infty 
e_\Phi\!\left(\sum_{i=q_n+1}^{q_{n+1}}
x^{(j_{q_n+1,n})}_i e_i\right)
\le \sum_{n=1}^\infty \frac{1}{2^n}
< \infty,
\]

so $x^\ast\in \ell_\Phi$.

But also,

\[
e_\Phi\!\left(\frac{x^\ast}{I}\right)
= \sum_{n=1}^\infty 
e_\Phi\!\left(\sum_{i=q_n+1}^{q_{n+1}}
\frac{x^{(j_{q_n+1,n})}_i}{I} e_i\right)
\ge \sum_{n=1}^\infty \frac12
= \infty,
\]

so $x^\ast\notin h_\Phi$.

Thus $\ell_\Phi\neq h_\Phi$ in Case I.

\bigskip
\noindent
\textbf{Case II.}  
Assume Case I fails.  
Then there exists $i_0$ such that

\[
\limsup_{n\to\infty}
\Phi_{i_0}\!\left(\frac{x^{(n)}_{i_0}}{I}\right) > 0.
\]

Choose a subsequence $\{x^{(n_k)}\}$ such that

\[
\Phi_{i_0}\!\left(\frac{x^{(n_k)}_{i_0}}{I}\right)
\to c_0 > 0.
\tag{16}
\]

From (9),

\[
\Phi_{i_0}\!\left(x^{(n_k)}_{i_0}\right)
\le e_\Phi(x^{(n_k)}) \to 0.
\tag{18}
\]

Let $r_k := x^{(n_k)}_{i_0}$.  
Then

\[
\Phi_{i_0}(r_k)\to 0,
\qquad
\Phi_{i_0}\!\left(\frac{r_k}{I}\right)\to c_0>0.
\tag{19}
\]

Since $\Phi_{i_0}$ is non--degenerate,  
$\Phi_{i_0}(t_n)\to 0$ implies $t_n\to 0$.  
Thus $r_k\to 0$, hence $r_k/I\to 0$, and by continuity,

\[
\Phi_{i_0}\!\left(\frac{r_k}{I}\right)\to 0,
\]

contradicting (19).

\bigskip
\noindent
Since Case II is impossible, Case I must occur, and in that case we constructed
$x^\ast\in \ell_\Phi\setminus h_\Phi$.  
Thus $h_\Phi\not\cong \ell_\Phi$.}
\end{proof}


{
\begin{lemma}\label{lemma:9}
Let $\Phi$ be a non--degenerate Musielak--Orlicz function.  
Then $h_\Phi \cong \ell_\Phi$ if and only if $\widetilde{\Phi}$ has a strong minimum at $0$.
\end{lemma}

\begin{proof}
\textbf{($\Rightarrow$)}  
Assume $h_\Phi \cong \ell_\Phi$.  
Since $\Phi_i(0)=0$ and $\Phi_i(t)>0$ for $t\neq 0$ (non--degeneracy), we have

\[
\widetilde{\Phi}(x)=\sum_{i=1}^\infty \Phi_i(x_i)\ge 0,
\qquad 
\widetilde{\Phi}(x)=0 \iff x=0.
\]

Thus $0$ is the unique minimizer of $\widetilde{\Phi}$.

Suppose $\widetilde{\Phi}$ does \emph{not} have a strong minimum at $0$.  
Then there exists a sequence $\{x^{(n)}\}\subset \ell_\Phi$ such that

\[
\widetilde{\Phi}(x^{(n)})\to 0
\quad\text{but}\quad
\|x^{(n)}\|_\Phi \not\to 0.
\]

Passing to a subsequence if necessary, we may assume

\[
\inf_{n\in\mathbb N}\|x^{(n)}\|_\Phi > 0.
\]

Since

\[
\widetilde{\Phi}(x^{(n)})=\sum_{i=1}^\infty \Phi_i(x^{(n)}_i)\to 0,
\]
Lemma \ref{lemma:8A} applies and yields an element $x\in \ell_\Phi$ with $x\notin h_\Phi$.  
This contradicts $h_\Phi\cong \ell_\Phi$.  
Hence $\widetilde{\Phi}$ must have a strong minimum at $0$.

\bigskip
\textbf{($\Leftarrow$)}  
Assume $\widetilde{\Phi}$ has a strong minimum at $0$.  
Suppose, toward a contradiction, that $h_\Phi\neq \ell_\Phi$.  
Then there exists $x^*\in \ell_\Phi$ with $x^*\notin h_\Phi$.  
Thus there exists $\lambda>0$ such that

\[
\widetilde{\Phi}(x^*)<\infty,
\qquad
\widetilde{\Phi}\!\left(\frac{x^*}{\lambda}\right)=\infty.
\]

Since

\[
\widetilde{\Phi}\!\left(\frac{x^*}{\lambda}\right)
= \sum_{i=1}^\infty \Phi_i\!\left(\frac{x^*_i}{\lambda}\right)
= \infty,
\]
and the partial sums are finite and increasing, for each $n\in\mathbb N$ there exists $m_n\ge n+1$ such that
\begin{equation}\label{eq:17}
1 \le 
\widetilde{\Phi}\!\left(
\frac{\sum_{i=n+1}^{m_n} x^*_i e_i}{\lambda}
\right)
< \infty.
\end{equation}

Define

\[
x^{(n)} := \sum_{i=n+1}^{m_n} x^*_i e_i \in \ell_\Phi.
\]

From \eqref{eq:17},

\[
\widetilde{\Phi}\!\left(\frac{x^{(n)}}{\lambda}\right)\ge 1.
\]

By the Luxemburg norm definition,

\[
\left\|\frac{x^{(n)}}{\lambda}\right\|_\Phi \ge 1
\quad\Longrightarrow\quad
\|x^{(n)}\|_\Phi \ge \lambda.
\]

Thus

\[
\inf_{n\in\mathbb N}\|x^{(n)}\|_\Phi \ge \lambda > 0.
\]

Since $\widetilde{\Phi}(x^*)<\infty$, the tail sums satisfy

\[
\sum_{i=n+1}^\infty \Phi_i(x^*_i)\xrightarrow[n\to\infty]{} 0.
\]

Hence

\[
\widetilde{\Phi}(x^{(n)})
= \widetilde{\Phi}\!\left(\sum_{i=n+1}^{m_n} x^*_i e_i\right)
\le \sum_{i=n+1}^\infty \Phi_i(x^*_i)
\xrightarrow[n\to\infty]{} 0.
\]

Thus we have constructed a sequence $\{x^{(n)}\}\subset \ell_\Phi$ such that

\[
\widetilde{\Phi}(x^{(n)})\to 0
\quad\text{and}\quad
\inf_n \|x^{(n)}\|_\Phi > 0,
\]

contradicting the assumption that $\widetilde{\Phi}$ has a strong minimum at $0$.

Therefore $h_\Phi = \ell_\Phi$.
\end{proof}
}


Let $\Phi$ be a Musielak-Orlicz function and $h_\Phi\cong \ell_\Phi$, then $\sup\left\{\widetilde{\Phi}(x):\|x\|_\Phi\leq t\right\}<\infty$ for any $t\geq 0$. Let us denote $\nu(t)=\sup\left\{\widetilde{\Phi}(x):\|x\|_\Phi\leq t\right\}$. 

\begin{lemma}\label{lemma:10}
Let $\Phi$ be a non--degenerate Musielak--Orlicz function and suppose $h_\Phi\cong \ell_\Phi$. 
Then for every $K>0$ the functional $\widetilde{\Phi}_a$ is bounded and Lipschitz on $K B_{\ell_\Phi}$, i.e.,
\begin{equation}\label{equation:18}
  \sup_{x\in K B_{\ell_\Phi}} \bigl|\widetilde{\Phi}_a(x)\bigr|
  \le \|a\|_\infty \,\nu(K),
\end{equation}
and

\[
  \bigl|\widetilde{\Phi}_a(x)-\widetilde{\Phi}_a(y)\bigr|
  \le \|a\|_\infty \,\nu(K+1)\,\|x-y\|_\Phi
\]

for every $x,y\in K B_{\ell_\Phi}$.
\end{lemma}

{
\begin{proof}
We proceed in two steps.

\medskip
\noindent\textbf{Step 1: $a\in \ell_\infty^+$.}
Assume first that $a\in\ell_\infty^+$, i.e.\ $a_i\ge 0$ for all $i\in\mathbb N$. Then

\[
  \widetilde{\Phi}_a(x)
  = \sum_{i=1}^\infty a_i \Phi_i(x_i)
  \ge 0,
\]

and

\[
  \widetilde{\Phi}_a(x)
  \le \|a\|_\infty \sum_{i=1}^\infty \Phi_i(x_i)
  = \|a\|_\infty \,\widetilde{\Phi}(x).
\]

Hence, for $x\in K B_{\ell_\Phi}$,

\[
  \widetilde{\Phi}_a(x)
  \le \|a\|_\infty \,\nu(K),
\]

which proves boundedness on $K B_{\ell_\Phi}$.

Since each $\Phi_i$ is convex and $a_i\ge 0$, the functional $\widetilde{\Phi}_a$ is convex.  
Thus it has directional derivatives $\widetilde{\Phi}_a'(x;h)$ for all $x,h$.

Let $x\in K B_{\ell_\Phi}$ and $h\in S_{\ell_\Phi}$ (i.e.\ $\|h\|_\Phi=1$).  
By convexity,

\[
  \widetilde{\Phi}_a(x+h)
  \ge \widetilde{\Phi}_a(x) + \widetilde{\Phi}_a'(x;h).
\]

Since $\|x\|_\Phi\le K$ and $\|h\|_\Phi=1$, we have $x+h\in (K+1)B_{\ell_\Phi}$, and hence

\[
  \widetilde{\Phi}_a(x+h)
  \le \|a\|_\infty \,\nu(K+1),
  \qquad
  \widetilde{\Phi}_a(x)\ge 0.
\]

Therefore

\[
  \widetilde{\Phi}_a'(x;h)
  \le \|a\|_\infty \,\nu(K+1)
\]

for all $x\in K B_{\ell_\Phi}$ and all $h\in S_{\ell_\Phi}$.

Applying the same argument with $-h$ instead of $h$ gives

\[
  \widetilde{\Phi}_a'(x;-h)
  \le \|a\|_\infty \,\nu(K+1),
\]

which implies

\[
  -\widetilde{\Phi}_a'(x;h)
  \le \|a\|_\infty \,\nu(K+1),
\]

so

\[
  \bigl|\widetilde{\Phi}_a'(x;h)\bigr|
  \le \|a\|_\infty \,\nu(K+1)
\]

for all $x\in K B_{\ell_\Phi}$ and all $h\in S_{\ell_\Phi}$.

By standard convex analysis, a convex functional whose directional derivatives are uniformly bounded on the unit sphere is Lipschitz with the same constant. Hence

\[
  \bigl|\widetilde{\Phi}_a(x)-\widetilde{\Phi}_a(y)\bigr|
  \le \|a\|_\infty \,\nu(K+1)\,\|x-y\|_\Phi,
  \qquad x,y\in K B_{\ell_\Phi}.
\]

\medskip
\noindent\textbf{Step 2: general $a\in\ell_\infty$.}
For arbitrary $a\in\ell_\infty$, write $a=a_+ - a_-$, where

\[
  (a_+)_i = \max\{a_i,0\},\qquad (a_-)_i = -\min\{a_i,0\},
\]

so that $a_+,a_-\in\ell_\infty^+$ and $\|a\|_\infty = \max\{\|a_+\|_\infty,\|a_-\|_\infty\}$.  
Then

\[
  \widetilde{\Phi}_a = \widetilde{\Phi}_{a_+} - \widetilde{\Phi}_{a_-}.
\]

Both $\widetilde{\Phi}_{a_+}$ and $\widetilde{\Phi}_{a_-}$ satisfy the boundedness and Lipschitz estimates from Step~1 with constants controlled by $\|a_+\|_\infty$ and $\|a_-\|_\infty$, respectively. Using

\[
  \|a_+\|_\infty,\ \|a_-\|_\infty \le \|a\|_\infty,
\]

we obtain the same type of bounds for $\widetilde{\Phi}_a$, with constants depending only on $\|a\|_\infty$. This proves the lemma.
\end{proof}

Set $\mathbb{P}=\left\{\widetilde{\Phi}_a:a\in \ell_\infty^{+}\right\}$,
where $\ell^+$ is the cone of all non-negative sequences.   We will consider the ${\rm span}\{\mathbb{P}\}$,
i.e., the set of all functions $\widetilde{\Phi}_a, a\in l_\infty$. Define a norm in ${\rm span}\{\mathbb{P}\}$ as $\| \widetilde{\Phi}_a\|_{\mathbb{P}}=\|a\|_\infty$. 
It is easy to see that $({\rm span}\{\mathbb{P}\},\|\cdot\|_{\mathbb{P}})$ and $(\ell_\infty,\|\cdot\|_\infty)$ are isometric.


\begin{lemma}\label{lemma:11}
    $(\{{\rm span} \mathbb{P}\},\|\cdot\|_{\mathbb{P}})$ is a perturbation space in the strong sense on each nonempty, closed and bounded subset $S$ of $\ell_{\Phi}$, provided that $h_\Phi=\ell_\Phi$.
\end{lemma}

{
\begin{proof}
    We verify the axioms of Definition~\ref{definition:11}, i.e.\ \ref{lab:d2-1}, \ref{lab:d2-2}, and \ref{lab:d8-3}, on an arbitrary ball $KB_{\ell_\Phi}$.
    Since $S$ is bounded, there exists $K>0$ such that $S\subset KB_{\ell_\Phi}$.

    By Lemma~\ref{lemma:10}, every function in ${\rm span}\{\mathbb{P}\}$ is continuous. 
    Assumption \ref{lab:d2-1} is immediate, and \ref{lab:d2-2} follows from~\eqref{equation:18}.

    Define 
    \[
        \varphi(t)=\operatorname{diam}\bigl(\Omega_{\widetilde{\Phi}}(t)\bigr).
    \]

    By Lemma~\ref{lemma:9},
    \begin{equation}\label{equation:19}
        \lim_{t\to 0}\varphi(t)=0 .
    \end{equation}

    To verify \ref{lab:d8-3}, fix $n\in\mathbb{N}$ and $\varepsilon>0$.  
    By~\eqref{equation:19}, choose $K_n>0$ such that
    \begin{equation}\label{equation:20}
        \varphi(3K_n)<\frac{1}{n}.
    \end{equation}
    Set $\delta=K_n\varepsilon$.  
    Choose $\theta\in(0,\varepsilon)$ so that
    \begin{equation}\label{equation:21}
        \theta\,\nu(K)<\frac{\delta}{2}.
    \end{equation}

    Let $x\in S$ be arbitrary; then $\|x\|\le K$.  
    Since $x\in\ell_\Phi$, its modular tails vanish, so we may choose $N\in\mathbb{N}$ such that
    \begin{equation}\label{equation:22}
        \varepsilon\,\widetilde{\Phi}\!\left(\sum_{i=N+1}^\infty x_i e_i\right)
        <\frac{\delta}{2}.
    \end{equation}
    Define

\[
        \overline{a}
        =\theta\sum_{i=1}^N e_i
        +\varepsilon\sum_{i=N+1}^\infty e_i ,
    \qquad
        g(y)=\widetilde{\Phi}_{\overline{a}}(y)
        =\theta\,\widetilde{\Phi}\!\left(\sum_{i=1}^N y_i e_i\right)
        +\varepsilon\,\widetilde{\Phi}\!\left(\sum_{i=N+1}^\infty y_i e_i\right).
    \]
    Then $\|g\|_{\mathbb{P}}=\|\overline{a}\|_\infty=\varepsilon$.  
    Moreover,

\[ \widetilde{\Phi}\!\left(\sum_{i=1}^N x_i e_i\right)
        \le \widetilde{\Phi}(x)
        \le \nu(K),
    \]
    so by~\eqref{equation:21},
    \begin{equation}\label{1-N}
        \theta\,\widetilde{\Phi}\!\left(\sum_{i=1}^N x_i e_i\right)
        <\frac{\delta}{2}.
    \end{equation}
    Combining \eqref{equation:22} and \eqref{1-N} yields

\[ g(x)<\delta,
\]
    so in particular \(x\in\Omega_g^S(\delta)\).

    Now let $y\in\Omega_g^S(3\delta)$, i.e.\ $\|y\|\le K$ and $g(y)\le 3\delta$.  
    Then

\[
        \varepsilon\,\widetilde{\Phi}\!\left(\sum_{i=N+1}^\infty y_i e_i\right)
        \le 3\delta
        =3K_n\varepsilon,
    \]

    hence

\[
        \widetilde{\Phi}\!\left(\sum_{i=N+1}^\infty y_i e_i\right)
        \le 3K_n,
    \]

    so

\[
        \sum_{i=N+1}^\infty y_i e_i \in \Omega_{\widetilde{\Phi}}(3K_n).
    \]

    Since $0\in\Omega_{\widetilde{\Phi}}(3K_n)$ and $\operatorname{diam}(\Omega_{\widetilde{\Phi}}(3K_n))=\varphi(3K_n)<1/n$ by~\eqref{equation:20}, we obtain

\[
        \left\|\sum_{i=N+1}^\infty y_i e_i\right\|<\frac{1}{n}.
    \]

    Define

\[
        A=\left\{\sum_{i=1}^N z_i e_i : \sum_{i=1}^\infty z_i e_i\in S\right\},
    \]
 which is compact because $S$ is closed and bounded.
  
    Because

\[
        y=\sum_{i=1}^N y_i e_i+\sum_{i=N+1}^\infty y_i e_i,
        \qquad
        \sum_{i=1}^N y_i e_i\in A,
    \]

    we have

\[
        \operatorname{dist}(y,A)
        \le \left\|\sum_{i=N+1}^\infty y_i e_i\right\|
        <\frac{1}{n}.
    \]

    Since $y\in\Omega_g^S(3\delta)$ was arbitrary, it follows that

\[
        \alpha\bigl(\Omega_g^S(3\delta)\bigr)\le \frac{1}{n},
    \]

    which is exactly condition~\eqref{equation:4}.  
    This completes the proof.
\end{proof}

\begin{lemma}\label{Lem-A}
	Let $\Phi:[0,+\infty)\to [0,+\infty)$ be convex and $\Phi(x)\ge 0$ for all $x\ge 0$. Then the inequality
	$$
	\left|\Phi(x_2)-\Phi(x_1)\right|\geq\Phi(x_2-x_1)
	$$
	holds for all $x_2\ge x_1 \ge 0$.
\end{lemma}

\begin{proof}
	Put $\lambda=\frac{x_2-x_1}{x_2}$. Then $x_2-x_1=\lambda x_2$ and by the convexity assumption on $\Phi$, we have
	$$
	\begin{array}{lll}
		\Phi(x_2-x_1)&=&\displaystyle\Phi(\lambda x_2+(1-\lambda)\cdot 0)\\
		&\leq& \lambda\Phi(x_2)\\
		&=&\frac{x_2-x_1}{x_2}\Phi(x_2)\\
		&=& \Phi(x_2)-\frac{x_1}{x_2}\Phi(x_2)\\
		&\leq&\displaystyle\Phi(x_2)-\Phi\left(\frac{x_1}{x_2}x_2+\left(1-\frac{x_1}{x_2}\right)\cdot 0\right)\\
		&=&\Phi(x_2)-\Phi(x_1).
	\end{array}
	$$ 
\end{proof}

For $x\in h_\Phi^+$ we introduce the notation
$$
\tilde x=\Big(\Phi_1(x_1),...,\Phi_n(x_n),...\Big)\in \ell_1^+
$$
and denote $h_\Phi^+=\{x\in h_\Phi : x_i\ge 0, i\in \mathbb{N}\}$.

\begin{lemma}\label{lemma:norms}
For any $x\in h_\Phi$ we have\ 
$
\|\tilde x\|_1\le \|x\|_{h_\Phi}.
$
\end{lemma}

\begin{proof}
For every $r\ge \|x\|_{h_\Phi}$ we have, by convexity of $\Phi_i, i\in \mathbb{N}$ and definition of Laxemberg norm
$$
{1\over r}\sum_{i=1}^\infty \Phi_i(x_i) \le \sum_{i=1}^\infty \Phi_i\Big ({x_i\over r}\Big ) \le 1.
$$
Hence, for  $r=\|x\|_{h_\Phi}$, we obtain
$$
\|\tilde x\|_1 := \sum_{i=1}^\infty \Phi_i(x_i) \le \|x\|_{h_\Phi}.
$$

\end{proof}

Define the mapping $T:h_\Phi^{+}\to\ell_1^{+}$ by $Tx=\widetilde{x}$.

\begin{cor}\label{cor-A}
If $\Phi$ is non-degenerate, then the mapping $T$ is a homeomorphism between  $\ell_\Phi^{+}$ and $\ell_1^{+}$.

\end{cor}

\begin{proof} Since $\Phi$ is non-degenerate, every $\Phi_i$ is strictly monotone, implying that  $T$ is one-to-one. 

(a) Let $\ell_1\ni \tilde x^k \rightarrow \tilde x^0\in \ell_1$. Lemma \ref{Lem-A} gives
$$
\sum_{i=1}^\infty \Phi_i(x_i^k-x_i^0)\le
\sum_{i=1}^\infty |\Phi_i(x_i^k) - \Phi_i(x_i^0)| \rightarrow 0.
$$
Now Lemma \ref{lemma:9} implies $x^k-x^0 \rightarrow 0$. 

(b) Let now $\ell_\Phi\ni x^k \rightarrow x^0\in \ell_\Phi$.
Lemma \ref{lemma:norms} shows that 
$\tilde x^k \rightarrow \tilde x^0$.
\end{proof}

}
 
\subsection{Monotone Mappings}

Let $X$ and $Y$ be Banach spaces such that $Y = X^*$ or $X = Y^*$.
We remind the following notion. 

\begin{definition}\label{definition:10}
	A multivalued mapping $F: X\rightarrow 2^Y$ is called monotone  if  
	\begin{equation}\label{equation:2}
		\langle x_1-x_2,y_1-y_2 \rangle \ge 0  \ \ \ \forall x_i\in X,\ \forall  y_i\in F(x_i), \ i=1,2,
	\end{equation}
	where $\langle.,.\rangle$ are the duality brackets between $X$ and $Y$.
\end{definition}

We will need the following theorem from   \cite{Ge}. 

\begin{theorem}(\cite{Ge}, Theorem 2.2)\label{theorem:6}. 
	Suppose $X$ and $Y$ are Banach spaces such that $Y = X^*$ or $X = Y^*$,  $Y$ is separable, and $T: X\rightarrow 2^{Y}$ is a locally bounded 
	monotone operator with domain ${\rm Dom}(T)=\{x\in X: T(x)\neq \emptyset\}\neq \emptyset$.
	Then there exists a {$\sigma$-fully cone porous} set $X_0\subset {\rm Dom}(T)$ such that $T$ is single-valued {and 
		upper semi-continuous} on ${\rm Dom}(T)\setminus X_0$.
\end{theorem}

Let $f:h_\Phi^+\supset S\rightarrow \mathbb{R}$ be a lower semi-continuous, proper and bounded from below function and  $\Phi=\{\Phi_n\}_{n=1}^\infty$ be a non-degenerate Musielak-Orlicz function. 

Denote
$$
\tilde S=\Big\{\tilde x\in\ell_1^+: x\in S\Big\},
$$
and define the function 
$$
\tilde f: \tilde S\rightarrow \mathbb{R} \ \ \ \mbox{by}\ \ \  \tilde f(\tilde x)=f(x).
$$
If $f$ is lower semi-continuous on $(S,\|.\|_{h_\Phi})$, then $\tilde f$ is lower semi-continuous on $(\tilde S, \|.\|_1)$ - this follow by  Corollary  {\ref{cor-A}.

 Define the multivalued mappings $F: \ell_\infty \rightarrow 2^S$ 
and $\tilde F: \ell_\infty \rightarrow 2^{\ell_1^+}$ respectively, as 
\begin{equation}\label{eq:*}
	F(a) = {\rm argmin}_{x\in S}  \left\{\sum_{i=1}^\infty  a_i\Phi_i(x_i)+ f(x) \right\}.
\end{equation}
and
\begin{equation}\label{eq:**}
	 \tilde F(a)={\rm argmax}_{\tilde x\in \tilde S}\Big \{\langle a,\tilde x \rangle -\tilde f(\tilde x)\Big \}.
\end{equation}

{
We have 
\begin{equation}\label{F}
x\in F(a) \ \ \ \mbox {if and only if} \ \ \  \tilde x\in \tilde F(-a).
\end{equation}
It is straightforward to check that 
$\tilde{F}: \ell_\infty \rightarrow 2^{\ell_1^+}$ is a monotone mapping.

}

{

\section{Main Results}
\begin{theorem}\label{theorem:13}
Let $\Phi$ be a non-degenerated Musielak--Orlicz function, assume $h_{\Phi}\cong \ell_{\Phi}$, 
and let $f:\ell_{\Phi}\to \mathbb{R}\cup\{+\infty\}$ be a proper, lower semicontinuous, 
and bounded below function whose domain
\[
S := \{x\in \ell_{\Phi} : f(x)<\infty\}
\]
is bounded.
%


Then

(a) There exists a set $A\subset \ell_\infty $ such that $\ell_\infty\setminus A$ is $\sigma$-porous and for every $a\in A$
the minimization problem $f+\widetilde{\Phi}_{a}$ is WPMC. 

(b) If $S\subset h_{\Phi}^{+}$, then 	
there  exists a  subset $A_0\subset A$ such that $\ell_\infty\setminus A_0$ is $  \sigma$-porous and such that
  the mapping $F: \ell_\infty \rightarrow \ell_\Phi, a\mapsto {\rm argmin}_{S}(f+\tilde \Phi_a) \ 
		$  is singled-valued and upper-semicontinuous on $\ell_\infty\setminus A_0$.
		
		(c) For every $a\in A_0$, every minimizing sequence for $f+\widetilde{\Phi_a}$ is norm converging to $F(a)$. 
 	
\end{theorem}

 
\begin{proof}

(a) 
By Lemma~\ref{lemma:11}, since $S$ is bounded, 
$({\rm span}\{\mathbb{P}\},\|\cdot\|_{\mathbb{P}})$ is a perturbation space in the strong sense for $S$.
Applying Theorem~\ref{theorem:7}, we obtain that the complement of the set

\[
\tilde A=\Bigl\{
   h\in {\rm span}\{\mathbb{P}\} :
   \text{the minimization problem } (f+h,S) \text{ is WPMC}
\Bigr\}
\]
 is a $\sigma$-porous subset of ${\rm span}\{\mathbb{P}\}$.

Since $({\rm span}\{\mathbb{P}\},\|\cdot\|_{\mathbb{P}})$ is isometric to 
$(\ell_\infty,\|\cdot\|_\infty)$, the isometric image in $\ell_\infty$ of $\tilde A$, denoted by $A$,  satisfies (a), i.e. $\ell_\infty \setminus A$ is $\sigma$-porous in $\ell_\infty$.

(b)  
{
Note that 
${\rm Dom}(\tilde F) \supset -A$.}

Since $\tilde F $ is a monotone operator,  Theorem~\ref{theorem:6} shows that it is single-valued and upper semicontinuous on a set
\[
A_1 \subset Dom(\tilde F),
\]
such that its complement in $Dom(\tilde F)$ 
is $\sigma$-fully cone porous. So, $\tilde F$ is single-valued and upper semicontinuous on the set 
$$
A_2:=(-A)\cap A_1, 
$$
which has a $\sigma$-porous complement in $\ell_\infty$. 
Due to Corollary \ref{cor-A}, the mapping  $F$ is single-valued and upper semicontinuous on $A_0:=-A_2$, which is also $\sigma$-porous.  

(c)  Since $F$ is also WPMC on $A_0$, the proof of (c) is completed.   

\end{proof}
  
 If the domain of $f$ is not bounded, we obtain similar result in the next theorem,
however the perturbation space is $\ell_\infty^+$.

\begin{theorem}\label{theorem:12}
Let $\Phi$ be a nondegenerate Musielak--Orlicz function and assume 
$h_\Phi \cong \ell_\Phi$.  
Let $f:\ell_\Phi^+\to\mathbb{R}\cup\{+\infty\}$ be proper, lower semicontinuous,
and bounded from below.  
For $a\in\ell_\infty^+$ and $x_0\in  \operatorname{Dom}(f)$ define
\[
F(a) := \operatorname{argmin}_{\ell_\Phi}\bigl(f+\widetilde{\Phi}_a\bigr).
\]
and
\[
P := 
\Bigl\{
a\in\ell_\infty^+ :
f+\widetilde{\Phi}_a(.-x_0) 
\text{ does not admit a strong global minimum on }\ell_\Phi^+
\Bigr\}.
\]

Then $P$ is $\sigma$--porous in $\ell_\infty^+$, and the mapping 
$F:\ell_\infty^+\setminus P\to\ell_\Phi$ is single--valued and 
upper--semicontinuous.
\end{theorem}

\begin{proof} (A) First assume that $x_0=0$.

\textbf{Case 1: $f$ is coercive.} So,

\[
\lim_{\|x\|_\Phi\to\infty} f(x)=+\infty.
\]

Choose $K>0$ such that
\begin{equation}\label{eq:coercive-K}
\|x\|_\Phi\ge K \quad\Longrightarrow\quad 
f(x) > f(0)+1.
\end{equation}
Set $S:=K B_{\ell_\Phi}\cap \ell_\Phi^+$.  
Since $S$ is bounded and $h_\Phi\cong\ell_\Phi$, Lemma~\ref{lemma:11} implies
that $\mathbb{P}$ is a strong perturbation space for $S$.

Define the ``bad'' set on $S$:

\[
A_S :=
\Bigl\{
a\in\ell_\infty^+ :
f+\widetilde{\Phi}_a 
\text{ does not admit a strong minimum on }S
\Bigr\}.
\]

By Theorem~\ref{theorem:13}, $A_S$ is $\sigma$--porous in $\ell_\infty^+$, and
for every $a\notin A_S$ the minimizer on $S$ is unique and depends 
upper--semicontinuously on $a$.

Fix $a\in\ell_\infty^+\setminus A_S$ and let 
$z\in S$ be the unique strong minimizer of $f+\widetilde{\Phi}_a$ on $S$.
 
If $x\not\in S$ and $x\in l_\Phi^+$, then $\|x\|>K$, and
	$$
	\begin{array}{lllr}
		f(x)+\widetilde{\Phi}_a(x)-1&\geq&f(x)-1&(\mbox{because}\ a\geq 0)\\
		&>&f(0)&(\mbox{by}\ (\ref{eq:coercive-K}))\\
		&\geq&f(z)+\widetilde{\Phi}_a(z) &(\mbox{since $\widetilde{\Phi}_a(0)=0$}).
	\end{array}
	$$
	So we proved that $z$ is a global minimum of the function $f+\widetilde{\Phi}_a$ on $\ell_\Phi^+$.
 
To see that the minimum is strong, let $\{x_n\}\subset \ell_\Phi^+$ be any minimizing sequence.
The above inequalities show that only finitely many  $x_n\not \in S$.
Hence $x_n\to z$.  
Thus $z$ is a strong global minimizer, and $F$ is single--valued and 
upper--semicontinuous on $\ell_\infty^+\setminus A_S$.

\medskip
\textbf{Case 2: $f$ is only bounded from below.}
For each $n\in\mathbb{N}$ define

\[
f_n := f + \frac{1}{n}\,\widetilde{\Phi}.
\]

Since $\widetilde{\Phi}(x)\to+\infty$ as $\|x\|_\Phi\to\infty$, each $f_n$ is
coercive.  
Applying Case~1 to $f_n$, define

\[
A_n :=
\Bigl\{
a\in\ell_\infty^+ :
f_n+\widetilde{\Phi}_a
\text{ does not admit a strong global minimum on }\ell_\Phi^+
\Bigr\}.
\]

Each $A_n$ is $\sigma$--porous in $\ell_\infty^+$.

We have the identity

\[
f_n+\widetilde{\Phi}_a
= f+\widetilde{\Phi}_{\,a+\frac{1}{n}\mathbf{1}},
\qquad \mathbf{1}=(1,1,\dots),
\]

Define

\[
B_n := A_n + \frac{1}{n}\mathbf{1}.
\]

Each $B_n$ is $\sigma$--porous.  
Moreover,
\begin{equation}\label{eq:por2-final}
b\in \ell_\infty^+ + \frac{1}{n}\mathbf{1}\ \setminus\ B_n
\quad\Longrightarrow\quad
f+\widetilde{\Phi}_b
\text{ has a strong global minimum on }\ell_\Phi^+.
\end{equation}

\medskip
\textbf{Claim.}
The set

\[
L :=
\ell_\infty^+ \setminus 
\bigcup_{n=1}^\infty 
\Bigl(\ell_\infty^+ + \tfrac{1}{n}\mathbf{1}\Bigr)
\]

is porous.

\emph{Proof.}
Let $B(x,r)\subset\ell_\infty^+$ be arbitrary.  
Define $y$ by $y_i=x_i+r/2$.  
Then $B(y,r/4)\subset B(x,r)$ and for any $z\in B(y,r/4)$,

\[
z_i\ge x_i+r/4\ge r/4.
\]

Thus $z\in \ell_\infty^+ + \frac{1}{n}\mathbf{1}$ for all $n>4/r$, so
$B(y,r/4)\cap L=\varnothing$.  
Hence $L$ is porous.  \qed

\medskip

Define the final exceptional set

\[
P := L \,\cup\, \bigcup_{n=1}^\infty B_n.
\]
Since $L$ is porous and each $B_n$ is $\sigma$--porous, the set $P$ is
$\sigma$--porous in $\ell_\infty^+$.

Now take $b\in\ell_\infty^+\setminus P$.  
Then $b\notin L$, so $b\in \ell_\infty^+ + \frac{1}{n_0}\mathbf{1}$ for some
$n_0$.  
Since $b\notin B_{n_0}$, property \eqref{eq:por2-final} implies that
$f+\widetilde{\Phi}_b$ has a strong global minimum on $\ell_\Phi$, and the
minimizer is unique and depends upper--semicontinuously on $b$.
 
(B) $x_0\neq 0$.

We apply case (A) for the function $g(x):=f(x+x_0)$ and finish the proof .
\end{proof}

The following theorem is similar to Theorem 4.4 from \cite{Topalova-Zlateva},  the proof is similar, using Lemma 10 and Theorem 6, and is omitted. 

\begin{theorem}\label{theorem 17}
Let $\Phi$ be a non-degenerate Musielak-Orlicz function, $\ell_\Phi=h_\Phi$, and let 
\(f:\ell_M \to \mathbb{R}\cup\{+\infty\}\) be a proper, lower semicontinuous, bounded below function whose domain is bounded.

Then for any \(0<\delta<\varepsilon\) there exists \(a=(a_n)\in \ell_\infty\) such that
\begin{equation}\label{eq:a-bounds}
\delta \le a_n \le \varepsilon,\qquad \forall n\in\mathbb{N},
\end{equation}
and the function \(f-\widetilde \Phi_a\) attains its minimum on \(\ell_M\).
\end{theorem}
} 
}   }
 
\section {Applications}

\subsection {The Stegall Variational Principle in Musielak-Orlicz Sequence Spaces}

\begin{theorem}\label{thm:generic-strong-min}
Let $X=\ell_{\Phi}$ be a Musielak--Orlicz sequence space, where $\Phi$ is non-degenerate and 
$h_{\Phi}\cong \ell_{\Phi}$. 
 Let $f:X\to\mathbb{R}\cup\{+\infty\}$ be proper, lower semicontinuous, bounded below, and with bounded domain
\[
S:=\{x\in X : f(x)<\infty\}.
\]

Then there exists a dense $G_\delta$ subset $X_0^*\subset X^*$ such that for every $x^*\in X_0^*$, 
the function $f-x^*$ attains a strong minimum on $X$.
\end{theorem}

{

\begin{proof}
Fix $\varepsilon>0$.  
Choose

\[
a\in A\cap B\left(0,\frac{\varepsilon}{\nu (K+1)}\right),
\]
where $K>0$ is such that $Dom(f)\subset KB_{h_\Phi}$ and $\ell_\infty \setminus A$ is a $\sigma$-porous exceptional set from
Theorem ~\ref{theorem:13}. Then $\|a\|\nu(K+1)< \varepsilon$.
  
By Theorem~\ref{theorem:13}, the 
minimization problem 
$$
  \mbox {minimize }\ \ f+\widetilde{\Phi}_a \ \ \ \mbox{over }\ S
$$
 is WPMC for every $a\in A$. Take any $a\in A$ and  $x_0\in \mbox{ argmin} f+\widetilde{\Phi}_a$.
 
 Since $f+\widetilde{\Phi}_a =f-\widetilde{\Phi}_{-a}$ 
and since $\widetilde{\Phi}_{-a}$ is convex, $\partial\widetilde{\Phi}_{-a}(x_0)\neq\emptyset$. Pick $x^*\in\partial\widetilde{\Phi}_{-a}(x_0)$.  
Then for all $x\in X$,
\begin{equation}\label{eq:subdiff-ineq}
   f(x)-f(x_0)
   \ge \widetilde{\Phi}_{-a}(x)-\widetilde{\Phi}_{-a}(x_0)
   \ge \langle x^*,x-x_0\rangle .
\end{equation}
which implies

\begin{equation}\label{24}
f(x)-f(x_0) - \langle x^*,x-x_0\rangle \ge f(x)-f(x_0) - 
( \widetilde{\Phi}_{-a}(x)-\widetilde{\Phi}_{-a}(x_0)) \ge 0.
\end{equation}
This shows that $x_0$ is a minimum point also for $f-x^*$.

Let $(x_n)$ be a minimizing sequence for $f-x^*$.  
From \eqref{24}, $(x_n)$ is also minimizing for $f+\widetilde{\Phi}_a$.  
The property WPMC for $f+\widetilde{\Phi}_a$ implies 
$$
dist(x_n,argmin_S(f+\widetilde{\Phi}_a)\to 0.
$$
Therefore 
$$
argmin_S(f+\widetilde{\Phi}_a)=argmin_S(f+x^*)
$$
and $f+x^*$ also has the property WPMC.

By \ref{eq:subdiff-ineq} and Lemma \ref{lemma:10}	

\[
\langle x^*,x-x_0\rangle \le |\tilde \Phi_a(x)-\tilde \Phi_a(x_0)| 
   \le \|a\|_\infty \nu(K+1)\,\|x-x_0\|_{\ell_\Phi}.
\]
which implies
 
\[
\|x^*\|\le \|a\|_\infty \nu(K+1) < \varepsilon.
\]

To obtain density in $X^*$, let $\xi^*\in X^*$ and $\varepsilon>0$.  
Apply the above argument to $f-\xi^*$ to obtain $y^*$ with $\|y^*\|<\varepsilon$ such that 
$f-\xi^*-y^*$ has the property WPMC.  
Set $x^*:=\xi^*+y^*$; then $\|x^*-\xi^*\|<\varepsilon$ and $f-x^*$ has the property WPMC.

For $n\in\mathbb{N}$ define

\[
X_n^*
=
\left\{
x^*\in X^*:\ \exists t>0\ \text{such that}\ 
\alpha	\big(\Omega^S_{f-x^*}(t)\big)<\frac{1}{n}
\right\}.
\]
Let $\xi^*\in X_n^*$ and $\beta>0$ be such that 
$$
\alpha	\big(\Omega^S_{f-\xi^*}(3\beta)\big)<\frac{1}{n}
$$
For any $y^*\in X^*$ with $\|y^*\|<\beta/diam(S)$ it follows that, for any $x, y\in S$,
$$
\langle y^*,x-y\rangle < \|y^*\|diam (S) < \beta
$$
so $\Omega^S_{y^*}(\beta)=S$. Applying Proposition \ref{proposition:1}, we obtain
$$
\Omega^S_{f-\xi^*-y^*}(\beta) \subset \Omega^S_{f-\xi^*}(3\beta)
$$
which shows that $\alpha(\Omega^S_{f-\xi^*-y^*}(\beta)) < 1/n$, therefore $\xi^*+y^*\in X_n^*$. So $X^*_n$ is open.

Each $X_n^*$ is open and dense, hence by the Baire category theorem,
\[
G^*:=\bigcap_{n=1}^\infty X_n^*
\]
is a dense $G_\delta$ subset of $X^*$, and for every $Gx^*\in G^*$ the minimization problem $f-x^*$ over $S$ is WPMC.

Now, since $\ell_\Phi=h_\Phi$, this space is separable and  by Theorem 16 applied again we obtain another dense $G_\delta$ subset $X^*_0\subset G^*$ such that $f-x^*$ attains its strong minimum on $S$ for every $x^*\in X^*_0$ and the multivalued mapping 
$$
G^*\ni x^* \rightarrow argmin_S (f-x^*)
$$
 is single valued and upper semi-continuous on $X^*_0$. 
\end{proof}

}

\begin{cor} \label{corollary:15}
	Let $\ell_{\Phi}$ be a Musielak-Orlicz sequence space, equipped with the Luxemburg's norm,  $\Phi$ be a non-degenerated Musielak-Orlicz function, $h_\Phi\cong \ell_\Phi$, then $\ell_\Phi$ is a dentable space, has the Radon-Nikodym property and $\ell_\Phi^*$ is a $w^*$-Asplund space.
\end{cor}
\begin{proof}
	The assertions follow from Theorem 5 in \cite{Lassonde}, where the reader can find also definitions of the involved notions.
\end{proof}

{Let us comment the well known fact that if a Banach space is reflexive, then it satisfies the Radon-Nikodym property \cite{Diestel1976RMJ}. Thus, any Musielak-Orlicz sequence space $\ell_{\Phi}$ has the Radon-Nikodym property, provided that, both, the generating function $\{\Phi_n\}_{n=1}^\infty$ and its conjugate one $\{\Phi_n^{*}\}_{n=1}^\infty$ satisfy the $\delta_2$ condition at zero. Recently, it was proven that for Orlicz sequence spaces $\ell_M$ (and Orlicz function spaces too) the assumption that $M\in\Delta_2$ at zero is necessary and sufficient for the space $\ell_M$ to be Radon-Nikodym \cite{KAMINSKA2020848}. It turns out, that all ''nice`` geometric properties in the Orlicz sequence spaces that are exposed in the presence of the $\Delta_2$ condition are also valid for the Musielak-Orlicz sequence spaces, if the $\delta_2$ holds. It is known that if there is an isomorphic copy of $c_0$ in $X$, then $X$ does not satisfy the Radon-Nikodym property. From \cite{Alherk1994} it follows that 
a Musielak-Orlicz sequence space $\ell_{\Phi}$ has an isomorphic copy of $c_0$ if and only if $\Phi$ does not satisfy the $\delta_2$ condition ate zero. Combining the result from \cite{Alherk1994} and Corollary \ref{corollary:15} we get a necessary and sufficient condition for the Radon-Nikodym property in Musielak-Orlicz sequence spaces is the validity of $\delta_2$ condition at zero.}

\begin{theorem}\label{theorem:17}
	Let $\Phi$ be a Musielak-Orlicz function and $1< \alpha_\Phi<2$. 
	
	Then
	\begin{enumerate}
		\item there is an equivalent Fr\'{e}chet differentiable norm, and therefore, $\ell_\Phi$ is an Asplund space 
		\item {If, moreover,  $h_\Phi\cong\ell_\Phi$,  and there are a sequence of naturals
			$\{k_i\}_{n=1}^\infty$ and a sequence $\{t_n\}_{m=1}^\infty$, convergent to zero, so that for some $p\in (1,2]$ and  every $i$ and $n$, 
			 \[ \displaystyle
  \frac{\Phi_{k_i}(t_n)}{t_n^p}>n, \mbox {and} \ \ 
  C:=\sup_{i\in \mathbb{N}}\Big [\Phi_i^{-1}(1)\Big ]^{-1} < +\infty, 
\] 
  then there is no nontrivial Fr\'{e}chet differentiable bump function on $\ell_\phi$, $b:\ell_\Phi\to\mathbb{R}$ with an estimate
		$$
		b(x+h)=b(x)+b^{\prime}(x)(h)+O(\|h\|^p).
		$$ 		
		}
\end{enumerate}

\end{theorem}

{
  
\begin{proof}[Proof of (2)]
Suppose, toward a contradiction, that there exists a nontrivial
Fr\'echet differentiable bump function $b:\ell_\Phi\to\mathbb R$
satisfying
\begin{equation}\label{eq:bump-estimate}
  b(x+h)=b(x)+b'(x)(h)+O(\|h\|^p),
  \qquad h\to 0,
\end{equation}
for some $p\in(1,2]$.

Define

\[
  f(x):=
  \begin{cases}
    b(x)^{-2}, & b(x)\neq 0,\\[1mm]
    +\infty, & b(x)=0.
  \end{cases}
\]

Then $f$ is proper, lower semicontinuous, positive, and Fr\'echet
differentiable on its domain

\[
dom f=\{x\in\ell_\Phi : b(x)\neq 0\},
\]

which is bounded because $b$ is a bump.

From \eqref{eq:bump-estimate} and the chain rule one obtains, for each
$x\in dom f$,

\[
  f(x+h)+f(x-h)-2f(x)=O(\|h\|^p),
\]

hence
\begin{equation}\label{eq:f-sym-final}
  \limsup_{\|h\|\to 0}
  \frac{f(x+h)+f(x-h)-2f(x)}{\|h\|^p}<\infty,
  \qquad \forall x\in dom f.
\end{equation}

Applying the variational principle to $f$, we obtain a sequence
$a=(a_n)\in\ell_\infty^+$ with $a_n\ge 1$ and a point
$\bar x\in  dom f$ such that $f-g_a$ attains its minimum at $\bar x$,
where

\[
  g_a(x):=\sum_{n=1}^\infty a_n\Phi_n(|x_n|).
\]

Thus
\begin{equation}\label{eq:min-ineq-final}
  f(x)-f(\bar x)\ge g_a(x)-g_a(\bar x),
  \qquad \forall x\in\ell_\Phi.
\end{equation}

Taking $x=\bar x\pm h$ in \eqref{eq:min-ineq-final} and adding the two
inequalities yields

\[
  f(\bar x+h)+f(\bar x-h)-2f(\bar x)
  \;\ge\;
  g_a(\bar x+h)+g_a(\bar x-h)-2g_a(\bar x).
\]

Dividing by $\|h\|^p$ and using \eqref{eq:f-sym-final} at $x=\bar x$
gives
\begin{equation}\label{eq:g-sym-bounded-final}
  \limsup_{\|h\|\to 0}
  \frac{g_a(\bar x+h)+g_a(\bar x-h)-2g_a(\bar x)}{\|h\|^p}<\infty.
\end{equation}

For each pair $(n,k)$ define $h_{n,k}:=t_n e_k$. We now use the
Luxemburg norm in $\ell_\Phi$:

\[
  \|t e_k\|_\Phi = \frac{t}{\Phi_k^{-1}(1)}.
\]

Hence

\[
  \|h_{n,k}\|
  = \|t_n e_k\|
  = \frac{t_n}{\Phi_k^{-1}(1)}.
\]

Moreover,

\[
  g_a(\bar x+h_{n,k})
  =a_k\Phi_k(|\bar x_k+t_n|),\qquad
  g_a(\bar x-h_{n,k})
  =a_k\Phi_k(|\bar x_k-t_n|),
\]

and

\[
  g_a(\bar x)=a_k\Phi_k(|\bar x_k|).
\]

Thus
\begin{equation}\label{eq:g-sym-nk-final}
  g_a(\bar x+h_{n,k})+g_a(\bar x-h_{n,k})-2g_a(\bar x)
  =
  a_k\Big(
    \Phi_k(|\bar x_k+t_n|)
    +\Phi_k(|\bar x_k-t_n|)
    -2\Phi_k(|\bar x_k|)
  \Big).
\end{equation}

By convexity of $\Phi_k$ we have
\[
  \Phi_k(|\bar x_k+t_n|)+\Phi_k(|\bar x_k-t_n|)
  \ge 2\Phi_k(t_n),
\]

and consequently
\begin{equation}\label{eq:g-lower}
  \Phi_k(|\bar x_k+t_n|)
  +\Phi_k(|\bar x_k-t_n|)
  -2\Phi_k(|\bar x_k|)
  \ge
  2\Phi_k(t_n)-2\Phi_k(|\bar x_k|).
\end{equation}

By Lemma~\ref{lemma:norms} we have

\[
  \Phi_k(|\bar x_k|)\longrightarrow 0.
\]

Fix $n$. Since $\Phi_k(|\bar x_k|)\to 0$, we may choose $k_{i_n}$ so
large that

\[
  \frac{\Phi_{k_{i_n}}(|\bar x_{k_{i_n}}|)}{t_n}<\frac1n.
\]
By hypothesis, for every $i$ and every $n$,

\[
  \frac{\Phi_{k_i}(t_n)}{t_n^p}>n,
\]
so in particular

\[
  \frac{\Phi_{k_{i_n}}(t_n)}{t_n^p}>n.
\]
Now use $a_k\ge 1$, \eqref{eq:g-sym-nk-final}, and \eqref{eq:g-lower}:

\[
\begin{aligned}
  &\frac{
    g_a(\bar x+h_{n,k_{i_n}})
    +g_a(\bar x-h_{n,k_{i_n}})
    -2g_a(\bar x)
  }{\|h_{n,k_{i_n}}\|^p}
  \\[1mm]
  &\qquad\ge
  \frac{
    2\Phi_{k_{i_n}}(t_n)-2\Phi_{k_{i_n}}(|\bar x_{k_{i_n}}|)
  }{\|h_{n,k_{i_n}}\|^p}
  \\
  &\qquad=
  \frac{
    2\Phi_{k_{i_n}}(t_n)-2\Phi_{k_{i_n}}(|\bar x_{k_{i_n}}|)
  }{\big(t_n/\Phi_{k_{i_n}}^{-1}(1)\big)^p}
  \\
  &\qquad=
  \Big(\Phi_{k_{i_n}}^{-1}(1)\Big)^p
  \frac{
    2\Phi_{k_{i_n}}(t_n)-2\Phi_{k_{i_n}}(|\bar x_{k_{i_n}}|)
  }{t_n^p}.
\end{aligned}
\]
By the assumption

\[
  C:=\sup_{i\in\mathbb N}\Big[\Phi_i^{-1}(1)\Big]^{-1}<\infty,
\]
we have

\[
  \Phi_{k_{i_n}}^{-1}(1)\ge \frac{1}{C}
  \quad\Rightarrow\quad
  \big(\Phi_{k_{i_n}}^{-1}(1)\big)^p \ge \frac{1}{C^p}.
\]
Hence

\[
\begin{aligned}
  &\frac{
    g_a(\bar x+h_{n,k_{i_n}})
    +g_a(\bar x-h_{n,k_{i_n}})
    -2g_a(\bar x)
  }{\|h_{n,k_{i_n}}\|^p}
  \\[1mm]
  &\qquad\ge
  \frac{1}{C^p}
  \frac{
    2\Phi_{k_{i_n}}(t_n)-2\Phi_{k_{i_n}}(|\bar x_{k_{i_n}}|)
  }{t_n^p}
  \\
  &\qquad=
  \frac{2}{C^p}\frac{\Phi_{k_{i_n}}(t_n)}{t_n^p}
  -\frac{2}{C^p}\frac{\Phi_{k_{i_n}}(|\bar x_{k_{i_n}}|)}{t_n^p}
  \\
  &\qquad>
  \frac{2}{C^p}n
  -\frac{2}{C^p}\cdot\frac{1}{n}
  \xrightarrow[n\to\infty]{}\infty.
\end{aligned}
\]
Since $t_n\to 0$, we have $\|h_{n,k_{i_n}}\|\to 0$, and therefore the
above divergence contradicts \eqref{eq:g-sym-bounded-final}. This shows
that no such Fr\'echet differentiable bump function can exist on
$\ell_\Phi$ with the estimate \eqref{eq:bump-estimate}.
\end{proof}

}

\begin{cor}\label{corollary:18}
	Let $\ell_{\{p_n\}}$ be a Nakano sequence space, $p_n\geq 1$, and
	$\liminf_{n\to\infty}p_n{=p}>1$. Then 
	\begin{enumerate}
		\item[\label=(a)] there is an equivalent Fr\'{e}chet differentiable norm in $\ell_{\{p_n\}}$ and, therefore, $\ell_{\{p_n\}}$ is an Asplund space
		\item[\label=(b)] if more over $\limsup_{n\to\infty}p_n<+\infty${ and $p<2$}, 	
		then on $\ell_{\{p_n\}}$ there is no nontrivial Fr\'{e}chet differentiable bump function $b:\ell_{\{p_n\}}\to\mathbb{R}$ with an estimate
		$$
		b(x+h)=b(x)+b^{\prime}(x)(h)+O(\|h\|^{p+\varepsilon})
		$$ 
		for any $\varepsilon >0$.
	\end{enumerate}
	
\end{cor}
\begin{proof}
	(a) follows from Theorem \ref{theorem:4}.
	
	The Nakano sequence space $\ell_{\{p_n\}}$ is a special case of a 
	Musielak-Orlicz sequence space $\ell_{\widetilde{\Phi}}$, with $\Phi_k(t)=t^{p_k}$.
	
	For (b) $\limsup_{n\to\infty}p_n<+\infty$ is equivalent to $h_{\{p_n\}}\cong\ell_{\{p_n\}}$
	and we will apply Theorem \ref{theorem:17}.
	
	Let us choose a subsequence $\{p_{n_k}\}_{k=1}^\infty$ so that $\lim_{k\to\infty}p_{n_k}=p$ and then choose the sequence $\{e_{n_k}\}_{k=1}^\infty$. For any arbitrary chosen $\varepsilon>0${, satisfying $\varepsilon\in (0,2-p)$,} there is $k_0\in\mathbb{N}$ so that for any $k\geq k_0$ there holds $p_{n_k}-p-\varepsilon<0$. Thus, for any arbitrary sequence $\{t_k\}_{k=1}^\infty$, verifying $t_k>0$ and $\lim_{k\to\infty}t_k=0$, we get
	$$
	\lim_{k \to \infty}\frac{\widetilde{\Phi}\left({t_ke_{n_k}}\right)}{|t_k|^{p+\varepsilon}}
	=\lim_{k \to \infty}\frac{{t_k^{p_{n_k}}}}{|t_k|^{p+\varepsilon}}
	=\limsup_{k \to \infty}t_k^{p_{n_k}-p-\varepsilon}=+\infty.
	$$
\end{proof}

{\begin{cor}\label{corollary:19}
	Let $M$ be an Orlicz function, $w=\{w_n\}_{n=1}^\infty$ be a wight sequence from the class $\Lambda_\infty$, and $1< \alpha_M<2$. Then 
	\begin{enumerate}
		\item there is an equivalent Fr\'{e}chet differentiable norm and $\ell_M(w)$ is an Asplund space
		\item if more over $h_M(w)\cong\ell_M(w)$ and there are a sequence of naturals
		$\{k_i\}_{n=1}^\infty$ and a sequence $\{t_n\}_{m=1}^\infty$, convergent to zero, so that for some $p\in (1,2]$ and  every $i$ and $n$, 
		$$
		w_{k_i}\frac{M(t_n)}{t_n^p}>n\ \mbox{and}\ C:=\sup_{i\in \mathbb{N}}\left[M^{-1}(1/w_i)\right]^{-1}<+\infty,
		$$ 
		then on $\ell_M(w)$ there is no nontrivial Fr\'{e}chet differentiable bump function $b:\ell_M(w)\to\mathbb{R}$ with and estimate
		$$
		b(x+h)=b(x)+b^{\prime}(x)(h)+O(\|h\|^p).
		$$ 
	\end{enumerate}
\end{cor}
\begin{proof}
	The weighted Orlicz sequence space $\ell_{M}(w)$, with $w\in\Lambda_\infty$ is a special case of a 
	Musielak-Orlicz sequence space $\ell_{\widetilde{\Phi}}$, with $\Phi_k(t)=w_kM(t)$.
	
	Since $h_M(w)\cong\ell_M(w)$, we can apply Theorem \ref{theorem:17}.
\end{proof}}

{\begin{cor}\label{corollary:20}
	Let $M$ be an Orlicz function, $w=\{w_n\}_{n=1}^\infty$ be a weight sequence from the class $\Lambda$,
	and $1< \alpha_M<2$.  Then 
	\begin{enumerate}
		\item there is an equivalent Fr\'{e}chet differentiable norm
		\item if moreover, $h_M(w)\cong\ell_M(w)$ and there is 
		a sequence $\{t_n\}_{m=1}^\infty$, convergent to zero, so that for some $p\in (1,2]$
		$$
		\frac{M(t_n)}{|t_n|^p}>n,
		$$ 
		then on $\ell_M(w)$ there is no nontrivial Fr\'{e}chet differentiable bump function $b:\ell_M(w)\to\mathbb{R}$ with and estimate
		$$
		b(x+h)=b(x)+b^{\prime}(x)(h)+O(\|h\|^p).
		$$
	\end{enumerate} 
\end{cor}
\begin{proof}
	By the assumption $w\in\Lambda$ it follows that there is a subsequence $\{w_{n_k}\}_{k=1}^\infty$ so that
	$$
	\lim_{k\to\infty}w_{n_k}=0\ \mbox{and}\ \sum_{k=1}^\infty w_{n_k}=+\infty.
	$$ 
	
	WLOG, just for simplicity of the notations we may assume that there holds 
	$$
	\lim_{n\to\infty}w_{n}=0\ \mbox{and}\ \sum_{n=1}^\infty w_{n}=+\infty.
	$$ 
	
	{Indeed the space $\ell_M(v)$, where $v=\{w_{n_k}\}_{k=1}^\infty$ is an isometric subspace of $\ell_M(w)$ and if there is no a bump function with the prescribed property in $\ell_M(v)$, the will be not in $\ell_M(w)$, too}. 
	
	We can choose two sequences of natural numbers $\{p_k\}_{k=1}^\infty$ and $\{q_k\}_{k=1}^\infty$
	so that 
	$$
	1\leq p_1\leq q_1<p_2\leq q_2<\dots p_{k}\leq q_{k}<p_{k+1}\leq\cdots
	$$ 
	and
	$$
	1-\frac{1}{k+1}\leq \sum_{i=p_{k}}^{q_k}w_i<1,
	$$
	for $k\in\mathbb{N}$.
	
	By the assumption $\frac{M(t_n)}{|t_n|^p}>n$ it follows that we can enumerate the sequence $\{t_n\}_{n=1}^\infty$ as follows: $u_n=t_{2n}$, for $n\in\mathbb{N}$. Then there holds
	$$
	\frac{M(u_n)}{|u_n|^p}>2n
	$$
	for all $n\in\mathbb{N}$. 
	
	We get that the inequality
	$$
	\sum_{i=p_{k}}^{q_k}w_i\frac{M(u_n)}{2|u_n|^p}\geq \frac{M(u_n)}{2|u_n|^p}>n.
	$$
	holds true for all $k\in\mathbb{N}$. 
	
	Let us put 
	$$
	\Phi_k(t)=\sum_{i=p_{k}}^{q_k}w_iM(t).
	$$
	
	Each of the functions $\Phi_k$ is an Orlicz function and thus $\Phi=\{\Phi_k\}_{k=1}^\infty$ is a Musielak-Orlicz one. Let us put $\{v_k\}_{k=1}^\infty$ for the unit vector basis in $\ell_\Phi$, $\{e_i\}_{i=1}^\infty$ for the unit vector basis in $\ell_M(w)$, and $s_k=\sum_{i=p_k}^{q_k}e_i$ for a block basic sequence. There holds $\sum_{i=1}^\infty\alpha_i v_i\in\ell_\Phi$ if and only if $\sum_{i=1}^\infty\alpha_is_i$ is $\ell_M(w)$ with equality of the norms $\|\sum_{i=1}^\infty\alpha_i v_i\|_\Phi=\|\sum_{i=1}^\infty\alpha_is_i\|_{\ell_M(w)}$. Thus $\ell_\Phi$ isometric to a subspace of $\ell_M(w)$, because the closed linear span od the block basis sequence $\{s_i\}_{i=1}^\infty$ is an isomorphic subspace of $\ell_M(w)$ \cite{Lindenstrauss-Tzafriri}.
	
	By the construction there holds $\frac{\Phi_k(u_n)}{|u_n|^p}>n$ for every $k,n\in\mathbb{N}$.
	From
	$$
	0<\Phi_k\left(1\right)\leq \sum_{i=p_k}^{q_k}w_i\leq 1
	$$
	it follows that 
	$$
	1\leq \Phi_k^{-1}\left(1\right)
	$$
	and thus
	$$
	\sup_{k\in\mathbb{N}}\left[\Phi_k^{-1}\left(1\right)\right]^{-1}\leq 1<+\infty
	$$

	We can apply Theorem \ref{theorem:17}.
\end{proof}}

\section{An Illustrative Example of a Weighted Orlicz Sequence Space, generated by on Orlicz function without $\Delta_2$--Condition}

Weighted Orlicz sequence spaces are an important generalization of classical sequence spaces. 
	In order to apply the fundamental results from functional analysis it is necessary the considered 
	weighted Orlicz sequence space to share some of the good geometrical of the Orlicz spaces. Thus we end up with 
	the idea to search for spaces generated by an Orlicz functions without the $\Delta_2$ conditons, but satisfying $\ell_M(w)\equiv h_M(w)$. We will mention just a few spheres of applications of weighted Orlicz sequence spaces: the provide the right function space framework for solutions of certain differential equations with non-standard growth conditions \cite{Radulescu-Repovs}; due to their flexibility in modeling decay and growth of sequences, they can be applied in signal representation \cite{Donnini-Vinti}; wavelet approximations \cite{Krivoshein-Skopina}; Fourier transforms and summability \cite{Jia}.

Let $N$ be an Orlicz function, $\displaystyle\lim_{t\to 0}\displaystyle\frac{N(t)}{t}=0$,
$N(1)=1$ and $\{w_k\}_{k=1}^\infty, w_1=1,$ be a weight sequence
from the class $\Lambda_\infty$ such that
{\begin{equation}\label{equation:26-z} 
	\sum_{k=1}^\infty\displaystyle\frac{w_k}{w_{k+1}}<\infty.
\end{equation}}

Denote
$a_k=N^{-1}(1/w_k)$ and choose a sequence $\{b_k\}_{k=1}^\infty$
fulfilling $0<a_{k+1}<b_k<a_k<1$, $\sum_{k=1}^\infty \frac{N(b_k)}{N(a_k)}<\infty$, and $	\lim_{k\to\infty}\frac{b_{k}}{a_k}=0$.

We define the Orlicz function $M$ by: 
\begin{equation}\label{equation:26} 
	M(t)=\left\{
	\begin{array}{lll}
		N(t),&t\in [a_{k+1},b_k],& k\in\mathbb{N}\\
		l_k(t),&t\in [b_k,a_k],& k\in\mathbb{N} ,
	\end{array}\right.
\end{equation}
where the line $l_k$ is defined by
$l_k(t)=\displaystyle\frac{N(a_k)-N(b_k)}{a_k-b_k}(t-a_k)+N(a_k)$.

Throughout this paragraph by $M$ we denote the Orlicz function
defined in (\ref{equation:26}), {$\{w_k\}_{k=1}^\infty$ be the weighted sequence defined in (\ref{equation:26-z}),} and let $\{a_k\}_{k=1}^\infty$ be the sequence defined as $N^{-1}(1/w_k)$.

\begin{proposition}\label{proposition:21}(\cite{Zlatanov})
	The weighted Orlicz sequence space $\ell_M(w)$ satisfies $\ell_M(w)\cong h_M(w)$ and
	$$
	\sum_{k=1}^\infty \frac{b_k}{a_k}<\infty .
	$$
\end{proposition}

{Let us put $S(t)=t^pM(t)$ for some $p>0$.}

\begin{proposition}\label{proposition:22}
	{For the Orlicz $M$ and the weighted sequence$\{w_k\}_{k=1}^\infty$ let $\beta_n$ be the solution of the equation $w_nM(\beta_n)=\beta$, and let us put $w_n^\prime=\frac{w_n}{\beta_n^p}$.
	Then }the weighted Orlicz sequence space $\ell_S(w^\prime)$ satisfies $\ell_S(w^\prime)\cong h_S(w^\prime)$.
\end{proposition}

\begin{proof} {From Proposition \ref{proposition:21} there holds} $\ell_M(w)\cong h_M(w)$ and {consequently} it follows that $\ell_{\{w_nM\}}\cong h_{\{w_nM\}}$.
	Therefore the Musielak-Orlicz function $\{w_nM\}_{n=1}^\infty$ has the $\delta_2$ condition at zero.
	Thus there are $\beta>0$ and $c_n\geq 0$, $\sum_{n=1}^\infty c_n<+\infty$ and there holds the inequality
	\begin{equation}\label{equation:27}
		w_nM(2t)\leq Kw_nM(t)+c_n, 
	\end{equation}
	provided $t\in [0,\beta_n]$, where $\beta=w_nM(\beta_{n})$.
	
	{For any $t\in [0,\beta_n]$, there holds $\frac{t^p}{\beta_n^p}\leq 1$. Using this inequality, the representation of the function $S$ and the $\delta_2$ property we get the }chain of inequalities  
	$$
	\begin{array}{lll}
		\displaystyle\frac{w_n}{\beta_n^p}S(2t)&=&\displaystyle\frac{w_n}{\beta_n^p}2^pt^pM(2t)\leq \frac{2^pt^p}{\beta_n^p}(Kw_nM(t)+c_n)\\
		&=&\displaystyle2^pK\frac{w_n}{\beta_n^p}t^pM(t)+{\frac{2^pt^p}{\beta_n^p}}c_n=2^pK\frac{w_n}{\beta_n^p}S(t)+{2^p}c_n\\
		&=&C\displaystyle\frac{w_n}{\beta_n^p}S(t)+b_n,  
	\end{array}
	$$
	for $t\in [0,\beta_n]$, where $\beta=w_nM(\beta_{n})$, and $\{b_n\}_{n=1}^\infty\in\ell_1$ {with $b_n=2^pc_n$,}
	we conclude that the Musielak-Orlicz function $\{w_n^\prime S(t)\}_{n=1}^\infty$
	has the $\delta_2$ condition at zero and therefore $\ell_S(w^\prime)\cong h_S(w^\prime)$.
\end{proof}

\begin{proposition}\label{proposition:23}
	Let $N$ be a non degenerate Orlicz function, $\{a_k\}_{k=1}^\infty$, $\{b_k\}_{k=1}^\infty$ and $M$ be defined with respect to (\ref{equation:26}), and $\displaystyle \lim_{k\to \infty}\frac{a_k}{b_k}=\infty$.
	
	Then $\alpha_{M}=1$. 
	
	If in addition $\mu a_{k+1}=b_{k+1}$ for some $\mu>0$ and $N\not\in\Delta_2$ then $\beta_M=+\infty$
\end{proposition}
\begin{proof}
	Let us define $\lambda_k=a_k$ and $\displaystyle t_k=\frac{\gamma_k a_k +(1-\gamma_k)b_k}{a_k}$, for $\gamma_k\in[0,1]$. Then, 
	\begin{equation}\label{equation:28}
		\begin{array}{ll}
			&\displaystyle \frac{M(\lambda_k t_k)}{t_k^p M(\lambda_k)}=\frac{(1-\gamma_k)M(b_k)+\gamma_k M(a_k)}{M(a_k)}\left(\frac{\gamma_k a_k +(1-\gamma_k)b_k}{a_k}\right)^{-p}\\
			\geq &\displaystyle \left(\frac{\gamma^{1-1/p}_k a_k +\gamma^{-1/p}_k(1-\gamma_k)b_k}{a_k}\right)^{-p}.
		\end{array}
	\end{equation}
	
	We can calculate that if $a<p (a-b)$, and $0<b<a$, then $\displaystyle 0< \frac{b}{(p-1)(a-b)}<1$. By assumption $\displaystyle \lim_{k\to \infty}\frac{a_k}{b_k}=\infty$. Therefore, for every $p>1$, there is $N_p\in \mathbb{N}$, so that $\displaystyle 0< \frac{b_k}{(p-1)(a_k-b_k)}<1$, for any $k>N_p$. Let us substitute $\displaystyle \gamma_k=\frac{b_k}{(p-1)(a_k-b_k)}$ in (\ref{equation:28}), for $k>N_p$. Then, we obtain $\displaystyle \frac{M(\lambda_k t_k)}{t_k^p M(\lambda_k)}\geq C_p\frac{a_k^p b_k^{1-p}}{a_k-b_k}\geq C_p \left(\frac{a_k }{b_k}\right)^{p-1}$, where $C_p= (p-1)^{p-1}p^{-p}$. Then from  $\displaystyle \lim_{k\to \infty}\frac{a_k}{b_k}=\infty$, it follows that $\displaystyle \lim_{k\to \infty}\frac{M(\lambda_k t_k)}{t_k^p M(\lambda_k)}=\infty$ for every $p>1$. Therefore $\alpha_{M}\leq1$, but $M$ is a convex function, thus $\alpha_{M}=1$.
	
	Let $\mu>1$ be so that $\mu a_{k+1}=b_{k+1}$. Thus $a_{k+1},\mu a_{k+1}\in [a_{k+1},b_{k+1}]$. Then from $N\not\in\Delta_2$ and
	$$
	\frac{M(\mu a_{k+1})}{M(a_{k+1})}=\frac{N(\mu a_{k+1})}{N(a_{k+1})}
	$$
	it follows that $M\not\in\Delta_2$ and consequently $\beta_M=+\infty$.
	
\end{proof}

The considered Musielak-Orlicz sequence spaces in \cite{Zlatanov}, Proposition \ref{proposition:21}
do not satisfy $\alpha_M\in (1,2)$. Therefore we will slightly change the function $M$.

\begin{proposition}\label{proposition:24}
	Let $\varphi$ and $\psi$ be two Orlicz functions, such that $\varphi(t)=t^q\psi(t)$, for some $q>0$. 
	Then $\alpha_\varphi=\alpha_\psi+q$ and $\beta_\varphi=\beta_\psi+q$
\end{proposition}
\begin{proof}
	From
	$$
	\displaystyle\frac{\varphi(\lambda t)}{t^p \varphi(\lambda)}
	=\displaystyle\frac{\lambda^qt^q \psi(\lambda t)}{t^p \lambda ^q\psi(\lambda)}
	=\frac{\psi(\lambda t)}{t^{p-q}\psi(\lambda)}.
	$$ 
	we get the proof.
\end{proof}

\begin{example}
	Let $N(t)=t^2e^{-\frac{1}{2t}}$ and
	$w_k=\displaystyle\frac{1}{N(1/k^{2k})}$, $k\in\mathbb{N}$. We define the sequences
	$\{a_k\}_{k=1}^\infty$, $\{b_k\}_{k=1}^\infty$ by
	$a_k=\displaystyle\frac{1}{k^{2k}}$, $b_k=\displaystyle\frac{2}{(k+1)^{2(k+1)}}$.
\end{example}  
From Proposition \ref{proposition:21} it follows
that $\ell_M(w)\cong h_M(w)$ and therefore there is $\beta$ so that the Musielak-Orlicz functions 
$\Phi_n(t)=w_nM(t)$ satisfy (\ref{equation:27}) for $t\in [0,\Phi_n^{-1}(\beta)]$. 

Let us choose $p\in (0,1)$ and put $w_n^\prime=\frac{w_n}{\beta_n^p}$, where $\beta_n$ are the solutions of the equations $w_nM(\beta_n)=\beta$. Let put $S(t)=t^pM(t)$.

We get a weighted Orlicz sequence space $\ell_S(w^\prime)$, such that
by Proposition \ref{proposition:22} $\ell_S(w^\prime)\cong h_S(w^\prime)$, and by Proposition \ref{proposition:24} $\alpha_S=\alpha_M+p=p+1<2$, $\beta_S=\beta_M=+\infty$.

Therefore there is an equivalent Fr\'{e}chet differentiable norm and there is no
nontrivial Fr\'{e}chet differentiable bump function $b:\ell_M(w)\to\mathbb{R}$ with and estimate
$$
b(x+h)=b(x)+b^{\prime}(x)(h)+O(\|h\|^{1+p}).
$$

\section{Conclusion}\label{sec13}

We proved that the set of perturbations for which a perturbed lower semi-continuous function is WPMC contains not only dense $G_\delta$ set but is also a complement to a $\sigma$-porous set in the entire space of perturbations. In such a way we generalize an abstract variational principle in Banach spaces \cite{Topalova-Zlateva}.  Furthermore, the stronger idea of Tikhonov well posedness takes the role of WPMC if the space is a Musielak-Orlicz sequence space satisfying $\ell_\Phi\cong h_{\Phi}$ and the domain of the function under consideration is in the positive octane of the space. Additionally, the multi-valued mapping assigning a parameter to the solution set is single-valued and upper semi-countinuous on the complement of a $\sigma$-porous set. We provide a number of applications. The first is that by proving the validity of Stegall's variational principle, the Musielak-Orlicz sequence spaces are dentable and possess the Radon-Nikodym property. Consequently, their dual spaces are $w^*$-Asplund spaces. We established, 
by a different argument, a sufficient condition under which the sequence spaces of Musielak-Orlicz and Nakano type are Asplund ones. 
We presented applications assessing the type of smoothness of specific Musielak-Orlicz, Nakano, and weighted Orlicz sequence spaces. We show that an Orlicz function can be considered without the $\Delta_2$ condition by selecting a certain weighted sequence $\{w_n\}_{n=1}^\infty$ to obtain $\ell_M(w)\cong h_M(w)$ and apply the main conclusions.



{\it Author Contribution:} 
All authors contributed equally and significantly in writing this article and are listed in alphabetical order. All authors read
and approved the final manuscript.

\bibliographystyle{plain}
\bibliography{sn-bibliography-23-01-2026}

\end{document}